\documentclass[10pt,a4paper]{article}
\usepackage[T1]{fontenc}
\usepackage[utf8]{inputenc}

\usepackage{cite}

\usepackage{mathpazo} % math & rm
\usepackage[scaled]{helvet} % ss
\usepackage{courier} % tt
\normalfont

\usepackage{amsmath}
\usepackage{amssymb}
\usepackage{mathtools}
\usepackage{amsthm}
\usepackage{tikz-cd}
\usepackage[hidelinks]{hyperref}	
\usepackage{graphicx}
\usepackage{subfig}
\usepackage{siunitx}
\usepackage{authblk}	
\usepackage{pinlabel}
\newtheorem{theorem}{Theorem}[section]
\newtheorem{lemma}[theorem]{Lemma}
\newtheorem{proposition}[theorem]{Proposition}

\theoremstyle{definition}
\newtheorem{definition}[theorem]{Definition}
\newtheorem{condition}[theorem]{Condition}	\newtheorem{assumption}[theorem]{Assumption}

\newtheorem{remark}[theorem]{Remark}
\usepackage{xcolor}

\providecommand{\MSC}[1]{\par\smallskip\textbf{\textbf{MSC2020: }} #1}
\providecommand{\keywords}[1]{\par\smallskip\textbf{\textbf{Keywords: }} #1}

\usepackage{booktabs}

\usepackage{tikz}  % richard: for my diagrams ;)
\usetikzlibrary{matrix} % richard: matrices in diagrams
\pgfdeclarelayer{background} %%%% creating two layers for tikz pictures, also line below
\pgfsetlayers{background,main} %%
\usepackage{verbatim}
\usepackage[normalem]{ulem}
\numberwithin{equation}{section}
\usepackage{todonotes}
\usepackage[normalem]{ulem}
\usepackage{bbm}
\usepackage{cancel}

\usepackage[shortlabels]{enumitem}

\newcommand{\cA}{\mathcal{A}}

\newcommand{\cC}{\mathcal{C}}
\newcommand{\cD}{\mathcal{D}}
\newcommand{\cE}{\mathcal{E}}
\newcommand{\cF}{\mathcal{F}}
\newcommand{\cG}{\mathcal{G}}
\newcommand{\cH}{\mathcal{H}}
\newcommand{\cI}{\mathcal{I}}

\newcommand{\cL}{\mathcal{L}}

\newcommand{\cP}{\mathcal{P}}

\newcommand{\cX}{\mathcal{X}}

\newcommand{\bE}{\mathbb{E}}

\newcommand{\bN}{\mathbb{N}}

\newcommand{\bQ}{\mathbb{Q}}
\newcommand{\bR}{\mathbb{R}}

\newcommand{\bT}{\mathbb{T}}

\newcommand{\bfH}{\mathbf{H}}

\newcommand{\bfR}{\mathbf{R}}

\newcommand{\bfV}{\mathbf{V}}

\DeclareMathOperator*{\LIM}{LIM}

\newcommand{\vn}[1]{\left| \! \left| #1\right| \! \right|} 

\newcommand{\ip}[2]{\langle #1,#2\rangle}
\newcommand{\PR}{\mathbb{P}}
\newcommand{\bONE}{\mathbbm{1}}
\newcommand{\dd}{ \mathrm{d}}

\title{A large deviation principle for Markovian slow-fast systems}
\author{Richard C. Kraaij\thanks{Delft Institute of Applied Mathematics, Delft University of Technology, Van Mourik Broekmanweg 6, 2628 XE Delft, The Netherlands. \emph{E-mail address}: r.c.kraaij@tudelft.nl} \quad \quad Mikola C. Schlottke\thanks{Department of Mathematics and Computer Science, Eindhoven University of Technology. \emph{E-mail address}: mikola.schlottke@outlook.com}}
\begin{document}
	%%%%%%%%%%%%

	\maketitle
	
	{\abstract{
			
			We prove pathwise large deviation principles of slow variables in slow-fast systems in the limit of time-scale separation tending to infinity. In the limit regime we consider, the convergence of the slow variable to its deterministic limit and the convergence of the fast variable to equilibrium are competing at the same scale.
			The large deviation principle is proven by relating the large deviation problem to solutions of Hamilton-Jacobi-Bellman equations, for which well-posedness was established in the companion paper~\cite{KrSc19}.
			
			We cast the rate functions in action-integral form and interpret the Lagrangians in two ways. First, in terms of a double-optimization problem of the slow variable's velocity and the fast variable's distribution, similar in spirit to what one obtains from the contraction principle. Second, in terms of a principal-eigenvalue problem associated to the slow-fast system. The first representation proves in particular useful in the derivation of averaging principles from the large deviations principles.
			
			As main example of our general results, we consider empirical measure-flux pairs coupled to a fast diffusion on a compact manifold. We prove large deviations and use the Lagrangian in double-optimization form to demonstrate the validity of the averaging principle in this system.
			%%% Old below
			% 	As two examples of our general results, we consider a slow diffusion process coupled to a fast jump process and a the empirical measure-flux pair coupled to a fast diffusion on a compact manifold. We use the Lagrangian in double-optimization form to demonstrate the validity of the averaging principle in this system.
			%%% Old above.
			
			
			\MSC{primary 60F10; 49L25; secondary 60J35} 
			
			%	49L25  	Calc of variations: Viscosity solutions for Hamilton-Jacobi equation
			%	60J35  	Transition functions, generators and resolvents
			%	60F10  	Large deviations for stochastic processes

			%%% end abstract
	}}
	
	\begin{comment}
	%%%%%% old abstract

	\sout{We prove path-wise large deviation principles of slow variables in fully-coupled slow-fast systems in the limit of time-scale separation tending to infinity. In this limit regime, the convergence of the slow variable to its deterministic limit and the convergence of the fast variable to equilibrium are competing at the same scale.
	
	The statement is proven by relating the large deviation problem to solutions of Hamilton-Jacobi-Bellman equations, for which well-posedness was established in the companion paper~\cite{KrSc19}.
	
	\smallskip
	
	We cast the rate functions in action-integral form and interpret the Lagrangians in two ways: in terms of a double-optimization problem of the slow variable's velocity and the fast variable's distribution, and in terms of a principal-eigenvalue problem associated to the slow-fast system. These representations of the Lagrangians prove in particular useful in the derivation of averaging principles from the large deviations principles.
	
	\smallskip
	
	As two examples of our general results, we consider a slow diffusion process coupled to a fast jump process and a the empirical measure-flux pair coupled to a fast diffusion on a compact manifold. We use the Lagrangian in double-optimization form to demonstrate the validity of the averaging principle in this system.
	%%% end \sout
	}
	\end{comment}

	%\tableofcontents
	\hspace{4cm}

	%%%%%%%%%%%%
	\section{Introduction}
	%%%%%%%%%%%%
	\subsection{Markovian slow-fast systems}
	Systems with two or multiple time-scales are ubiquitous in the natural sciences and beyond. For example, such systems are studied in finance \cite{FeFoKu12,BaCeGh15,Gh18}, atmospheric models \cite{BoNaTa13,BoGrTaEi16}, the theory of hydrodynamic limits \cite{KL99}, and
	genetic networks~\cite{CrDeRa09}. 	They also arise in statistical physics for models at critical parameters \cite{AdMPe91,CoDaP12,CoKr18,CoKr20}. See also \cite{PaSt08} for a general mathematical treatment for two-scale diffusion processes. 
	
	The hallmark of systems with multiple time-scales is the equilibration of fast components on a time scale at which slow components have yet to make a significant change. As a consequence, the slow components evolve approximately under the averaged effect of the fast components. On the other hand, the equilibrium state of the fast components depends on the state of the slow components. This leads to an interesting interplay between slow and fast states. This observed separation of time scales and the resulting coupling motivates the term \emph{slow-fast system}.
	
	\smallskip
	
	In this paper, we focus on the evolution of the slow components of a slow-fast systems in a Markovian context. 
	The approximation of the dynamics of the slow components by averaging over the fast components is known as the \emph{averaging principle}. This procedure significantly reduces complexity and leads to models that are easier to analyze. In the limit in which the time-scale separation tends to infinity this transition can be justified by rigorous convergence results, as demonstrated for instance in~\cite{BaKuPoRe06,KaKu13,KaKuPo14} in the context of reaction networks.

	\smallskip

	In order to estimate the approximation error at a finite time-scale separation, many efforts have therefore concentrated on establishing finer asymptotic results. An example of such an asymptotic result is a path-wise large deviation principle for the slow component as the time-scale separation tends to infinity. The large deviation principle quantifies the decay of the probability of deviations away from the averaging principle at an exponential scale as a function of the time-scale separation. 
	
	The large deviation analysis is a crucial ingredient for analyzing the rare-event behaviour of the system, see Theorem 1.4 of~\cite{BuDu19}. In addition, the large deviation analysis can subsequently be used to design Monte-Carlo for estimating rare-event probabilities, as carried out in the context of multi-scale diffusions~\cite{DuSpWa12}. 
	
	\smallskip
	
	The analysis of the large deviation behaviour of slow-fast systems is typically carried out for systems that explicitly consist of a slow and a fast component, in contrast to the systems studied at criticality in statistical physics \cite{KL99,AdMPe91,CoDaP12,CoKr18,CoKr20}. This explicit decomposition has the benefit of having clear interpretability of results and serves as a precursor to the study of more complicated systems.
	
	We will therefore work in the context of slow-fast systems with a clear decomposition between slow and fast components. As mentioned above, the slow component possibly depends on the behaviour of the fast component. For clarity, we distinguish three levels of dependency:
	%%% Added double-points : .
	\begin{enumerate}[$\bullet$]
		\item \textit{Independent}: the fast component is independent of the slow one.
		\item \textit{Weak}: the fast component depends weakly on the slow component. This is e.g. observed in jump processes where the coefficients of the jump rates depend on the slow-component. Similarly, this occurs for diffusion processes where the drift and diffusion terms of the fast component depend on the slow component.
		\item \textit{Strong}: the jump sizes of the fast process themselves depend on the slow process, or the generator of a slow-fast system of diffusion type contains mixed derivates, i.e. arising from both slow-and fast components.
		
	\end{enumerate}
	%%% old below
	% \begin{itemize}
	%     \item \textit{slaving}: in this context the fast component is independent of the slow one.
	%     \item \textit{weak coupling}: the fast component depends weakly on the slow component. This is e.g. observed in jump processes where the coefficients of the jump rates depend on the slow-component. Similarly, this occurs for diffusion processes where the drift and diffusion terms of the fast component depend on the slow component.
	%     \item \textit{strong coupling}: the jump sizes of the fast process themselves depend on the slow process, or the generator of a slow-fast system of diffusion type contains mixed (i.e. arising from both slow-and fast components) derivatives.
	% \end{itemize}
	%%% old above
	We will work in the context of weak coupling. Let us first review the literature of the results obtained so far.
	
	\smallskip
	
	The first large deviations results for slow-fast systems were obtained by Wentzell (recorded e.g. in \cite{FW98}) for a slow system and an independent fast system. Other early works in the context of diffusion process in which the slow process is coupled to an independent fast one include  \cite{Li96,Ve99}. This has been recently also been carried out using Hamilton-Jacobi methods in \cite{FeFoKu12,BaCeGh15,Gh18}, as well as in the context where the fast process is replaced by a jump process \cite{HeYi14,HuMaSp16}.

	\smallskip
	
	In the context that the fast process depends weakly on the slow one progress was initially made for diffusion processes by \cite{Ve00} and this has been extended using the weak convergence method in \cite{Sp13} for diffusion processes and by \cite{BuDuGa18} for diffusion processes coupled to a fast jump process. A treatment of the behaviour of the empirical measure of weakly interacting jump processes coupled to a fast jump process has been given in \cite{YaSu20}. Finally, an interesting account \cite{BoGrTaEi16} was written from the point of view of statistical physics. A unified and effective treatment of these results is so far lacking from the literature. We comment more on this below. The results of this paper will focus also on this weakly coupled regime.
	
	\smallskip
	
	To be complete, we mention results in the strongly coupled regime. \cite{Pu16} studies fully coupled slow-fast diffusions by considering the joint distribution of the slow process and the empirical measure of the fast process. Using Hamilton-Jaccobi theory, \cite{FK06} establish large deviation results for strongly coupled diffusion processes under restrictions on the coefficients. This is similarly done by \cite{KuPo17,Po18}, but this time under the assumption of solvability of the Hamilton-Jacobi equation, which is a non-trivial problem in itself. 
	
	\smallskip
	
	Our results will be based in the context of the weakly coupled regime. We will introduce a method, based on the Hamilton-Jacobi technique, that allows us to treat slow-fast processes of any type (we assume compactness of the state-space of the fast process to avoid technicalities that are not of importance for our general argument) in the context that we know that 
	\begin{itemize}
		\item the slow process satisfies a Donsker-Varadhan large deviation principle if the fast process is kept fixed;
		\item the fast process satisfies path-space large deviations if the slow process is kept fixed;
		%%% added the ;
		\item the coefficient processes depend 
		weakly
		%%% (instead of) `weakly' 
		on each other.
	\end{itemize}
	This framework allows us to recover previous results obtained in the weak-coupling regime, given that the fast process takes its values in a compact state space. Moreover, we treat the context of the empirical measure-flux pair of weakly interacting jump processes coupled to a fast diffusion. The results by \cite{YaSu20} are extended by our method. Indeed, they are obtained when replacing the fast diffusion by a fast jump process and considering only the empirical measures instead of the more general measure-flux pairs.
	
	%we also establish these results in the context where the slow process is derived from a collection of weakly interacting jump processes. For such a collection there are no results available in the slow-fast context.
	% (given that the fast process takes its values in a compact state space) obtained in the weak-coupling regime and \mi{to} establish these \mi{in \sout{for}} the context where the slow process is derived from a collection of weakly interacting jump processes\mi{. For such a collection there are no results available in the slow-fast context.}
	%%% (instead of) for there are no available results in the slow-fast context. 
	
	\smallskip
	
	In our framework, we will reduce the large deviation principle to uniqueness of viscosity solutions to the Hamilton-Jacobi-Bellman equation in terms of an Hamiltonian that is given in variational form:
	\begin{equation} \label{eqn:Hamiltonian_intro}
		\cH(x,p) = \sup_{\pi} \left\{\int V_{x,p}(z) \pi(\dd z) - \cI(x,\pi) \right\}.
	\end{equation}
	In this context, $(x,p) \mapsto V_{x,p}(z)$ is the Hamiltonian that arises from the path-space large deviations of the slow process if the fast process is kept fixed at $z$. The map $\pi \mapsto \cI(x,\pi)$ is the Donsker-Varadhan rate function of the fast process when the slow process is kept fixed at $x$. The variational formula arises from an averaging procedure or a principle eigenvalue problem that supersedes the averaging principle for the limiting dynamics. Indeed, the averaging principle can be recovered by choosing $p = 0$, after which the optimal $\pi$ turns out to be the unique $\pi_x$ such that $\cI(x,\pi_x) = 0$.
	%%% (instead of footnote) \footnote{Indeed, the averaging principle can be recovered by choosing $p = 0$, after which the optimal $\pi$ turns out to be the unique $\pi_x$ such that $\cI(x,\pi_x) = 0$.}.
	
	Due to the variational form of $\cH$, establishing uniqueness of the Hamilton-Jacobi-Bellman equation in this context is a serious challenge as was pointed out by \cite{BuDuGa18}. In \cite{FK06} the uniqueness was established by an argument that allowed the reduction to the Hamilton-Jacobi equation in terms of $V$. This however was carried out in the context of quadratic Hamiltonians only, which therefore restricts the method to diffusion processes.
	
	\smallskip
	
	Our companion paper \cite{KrSc19} established uniqueness for the Hamilton-Jacobi-Bellman equations for a much wider class of operators of the type \eqref{eqn:Hamiltonian_intro}. In this paper, we use this result to obtain a large deviation principle with a similarly wide scope. To focus on the new ideas in this framework, we restrict ourselves to the setting where the state-space of the fast process is compact. We expect an extension beyond this case to be possible under the correct assumptions, but this will introduce additional complications that divert from the main message of the text.

	\smallskip
	
	We proceed our introduction with an example that illustrates our main results, Theorems \ref{theorem:LDP_general} and \ref{theorem:LDP_general_Lagrangian_two_represenatations}. We then discuss our setting and the main methods. 
	
	%%%%%%%%%
	%\section{An introduction of the key ideas}
	\subsection{An illustrating example}
	Consider the process
	\begin{equation*}
		\dd X_n(t)   = F(X_n(t)) \dd t + \frac{1}{\sqrt{n}} \, \dd B(t).
	\end{equation*}
	As $n \rightarrow \infty$, we have $X_n \rightarrow x$, where $x$ solves 
	\begin{equation} \label{eqn:solve_lim_differential_eq_intro}
		\dot{x}(t) = F(x(t)).
	\end{equation}
	Freidlin-Wentzell theory extends this statement and gives speed of decay for `non-typical' trajectories~$\rho$,
	\begin{equation*}
		\PR[X_n \approx \rho] \asymp e^{-nJ(\rho)},
	\end{equation*}
	where $J(x) = 0$ for $x$ solving \eqref{eqn:solve_lim_differential_eq_intro}. In addition, $J$ is given in terms of the speed:
	\begin{equation*}
		J(\rho) = \int_0^\infty \frac{1}{2} \left|\dot{\rho}(t) - F(\rho(t))\right|^2 \dd t.
	\end{equation*}
	Now, consider the process $(X_n(t),Z_n(t))$, where $X_n(t)$ is a diffusion process as above
	\begin{equation*}
		\dd X_n(t)  = F(X_n(t),Z_n(t)) \dd t + \frac{1}{\sqrt{n}} \, \dd B(t),
	\end{equation*}	
	and where $Z_n(t)$ is a jump process on $\{1,2\}$ switching its state at rate $n$.  As $Z_n(t)$ equilibrates at the measure $\frac{1}{2}(\delta_{1} + \delta_2)$, we have $X_n \rightarrow x$ where
	\begin{equation*}
		\dot{x}(t) = \overline{F}(x(t)) \dd t, \qquad \overline{F}(x) = \frac{1}{2} \left( F(x,1) + F(x,2)\right). 
	\end{equation*}
	When we consider the heuristics regarding the large deviation behaviour of $X_n(t)$, we observe the following two facts.
	
	\begin{itemize}
		\item We have Donsker-Varadhan large deviations for the occupation measures of $Z_n$:
		\begin{gather*}
			\PR\left[ \int_0^1 \delta_{Z_n(t)} \dd t \approx \pi \right] \asymp \exp \left\{-n I(\pi) \right\}, \\
			\cI(\pi) = - \inf_{\phi >> 0} \int \frac{A\phi}{\phi} \, \dd \pi, 
		\end{gather*}
		and where $A$ is the generator of a jump process that switches between the states $1$ and $2$ with rate $1$:
		\begin{equation*}
			A\phi(1) = \phi(2) - \phi(1), \qquad A\phi(2) = \phi(1) - \phi(2).
		\end{equation*}
		\item Suppose that by a large deviation $Z_n(t)$ remains stationary at some law $\pi = (\pi(1),\pi(2))$. Conditional on this event, we can compute the large deviation rate that $t \mapsto X_n(t)$ moves at speed $v_1$ while $Z_n$ is in state $1$ and at speed $v_2$ while $Z_n$ is in state $2$. Freidlin-Wentzel theory tells us that this probability is asymptotically given by
		\begin{multline*}
			\exp \left\{- \frac{n}{2} \left(\int_{t : Z_n(t) =1} \left( v_1 - F(X_n(t),1)\right)^2 \dd t \right. \right. \\
			\left. \left. + \int_{t : Z_n(t) =2 } \left( v_2 - F(X_n(t),2)\right)^2 \dd t \right) \right\}.
		\end{multline*}
		Due to the stationarity of the $Z_n$ process in $\pi$, we can rewrite the integral as
		\begin{equation*}
			\frac{1}{2} \int_0^\infty \pi(1)  \left( v_1 - F(X_n(t),1)\right)^2 + \pi(2) \left( v_2 - F(X_n(t),2)\right)^2 \dd t.
		\end{equation*}
		The process $t \mapsto X_n(t)$ effectively has speed $v = \pi(1) v_1 + \pi(2) v_2$.
	\end{itemize}
	We thus see that the Freidlin-Wentzell and Donsker-Varadhan large deviations compete at the same scale and leave room for 2 optimizations. To produce a speed $v$ for the process $X_n$, we can first choose the measure $\pi$ over which the fast process equilibrates, and afterwards we can assign a speed $v_1$ to produce while $Z_n$ is in $1$ and a speed $v_2$ while $Z_n$ is in state $2$ in such a way that $v = \pi(1)v_1 + \pi(2) v_2$.
	This leads to a large deviation principle
	\begin{equation*}
		\PR[X_n \approx \rho] \asymp e^{-n J(\rho)}
	\end{equation*}
	where $J$ is obtained by optimizing over the possible choices for the stationary background measure $\nu$ for $Z_n(t)$:
	\begin{equation} \label{eqn:intro_variational_Lagrangian2}
		J(\rho) =  \int_0^\infty \inf_{\substack{\pi, v_1,v_2 \\ \pi(1)v_1 + \pi(2) v_2 = \dot{\rho}(t)}} \left\{\frac{1}{2} \int \left|v_i - F(\rho(t),i))\right|^2 \pi(\dd z) + \cI(\pi)  \right\}\dd t.
	\end{equation}
	The above argument, except for the quadratic form of the Lagrangian appearing in the rate function and the form of the generator $A$ of the fast process, is completely general. The key components of the heuristic hinge on the path-space large deviation principle of the slow process and the Donsker-Varadhan large deviations of the fast process.
	%%%%%%%%%
	\subsection{Aim of this note: a general framework for two time-scale Markov processes}
	%%%%%%%%%
	We consider sequences of processes $(X_n(t),Z_n(t))$ with generators
	\begin{equation*}
		A_n f(x,z) = \left(A^\mathrm{slow}_{n,z} f(\cdot,z)\right)(x) + n \left(A^\mathrm{fast}_x f(x,\cdot)\right)(z). 
	\end{equation*}
	and suppose the following:
	\begin{enumerate}[(a)]
		\item If we fix $Z_n(t)$ at $z$, then we have path-space large deviations at speed $n$ for the process $X_n(t)$. \label{item:intro_LDP_paths}
		\item If we fix $X_n(t)$ at $x$, then we have Donsker-Varadhan at speed $n$ large deviations for $Z_n(t)$. \label{item:intro_LDP_DV}
	\end{enumerate}	
	This paper focuses on two goals: identifying sufficient conditions on top of~\ref{item:intro_LDP_paths} and~\ref{item:intro_LDP_DV} that lead to a large deviation principle for the processes $X_n(t)$, and establishing a Lagrangian rate function as in \eqref{eqn:intro_variational_Lagrangian2} in terms of an optimization procedure over measures that involves ergodic averages of the fast process.
	Our approach is based on the semigroup method by \cite{FK06}, see also \cite{Kr20,Kr19}. This approach is further detailed in Section~\ref{section:outline_of_proof} below.
	The main analytical challenge in carrying out the semigroup method is the well-posedness of associated Hamilton-Jacobi-Bellman equations. The equations arising in the context of slow-fast systems have been treated in our companion paper~\cite{KrSc19}. In this paper we therefore focus on the probabilistic context and the reduction to the Hamilton-Jacobi-Bellman equation treated in \cite{KrSc19}.

	\subsection{A case-study: Large deviations for weakly interacting jump processes}
	
	To show the applicability of our two main results, Theorems \ref{theorem:LDP_general} and \ref{theorem:LDP_general_Lagrangian_two_represenatations}, we work out their application in the context of weakly interacting jump processes.
	% 	coupled to a fast diffusion. 
	% 	\smallskip
	% 	While there is an enormous interest and huge literature on the topic of slow-fast systems, the important class of physical models of mean-field interacting particles described by jump processes has not been treated so far in the context of slow-fast systems. 
	These Markovian jump models are frequently consulted as approximations to physical models describing non-equilibrium phenomena, such as spin dynamics. An overview involving different spin models is offered for instance in~\cite{Ma99}. A typical example is the Glauber dynamics in Ising-models and Potts-models describing ferromagnets. Other fields of applications include communication networks~\cite{AnFrRoTi06}, game theory with models involving a large number of agents~\cite{GoMoRi10}, and chemical reactions~\cite{MiPaPeRe17}. 
	% 	While there is an enormous interest and huge literature on the topic of slow-fast systems, the important class of physical models of mean-field interacting particles described by jump processes has not been treated so far in the context of slow-fast systems.
	\smallskip
	
	There is huge interest in the study of Markovian mean-field jump processes from a path-space large-deviation perspective. Recent studies include large deviations of the empirical densities with more than one jump simultaneously via the weak convergence method~\cite{DuRaWu16}; the stability of the nonlinear limit evolution equation of the particle system by constructing Lyapunov functions from relative entropies~\cite{BuDuFiRa15a}; large deviations of density-flux pairs exploiting Girsanov transformations~\cite{Re18,PaRe19} and from a Hamilton-Jacobi point of view, including deterministic time-periodic rates, in \cite{Kr17}; and for the empirical measure of weakly interacting particles, without their fluxes, weakly coupled to a fast Markov jump process~\cite{YaSu20}. Finally, Donsker-Varadhan large deviations for the empirical-measure combined with the fluxes and modulated by deterministic time-periodic rates are treated in~\cite{BeChFaGa18}. 
	
	\smallskip
	
	To our knowledge, there are no results on the dynamic large deviation principles for the empirical measure-flux pair of mean-field interacting particles coupled to fast external processes. In our work, we fill this gap. To make the example concrete, we choose a diffusion process as the external process. The arguments can however also be carried out for a fast jump process, which extends the results of~\cite{YaSu20}.
	
	% However, there are few results illuminating the large-deviation behaviour of mean-field particles from a Hamilton-Jacobi point of view~\cite{Kr17}. 
	% In our work, we contribute to the very recent studies by proving dynamic large deviation principles for the empirical measure-flux pair of mean-field interacting particles coupled to fast external processes. To make the example concrete, we choose a diffusion process as the external process. The arguments can however also be carried out for a fast jump process, which extends the results of~\cite{YaSu20}. 
	% Another novel aspect of our paper lies in the fact that we establish the large deviation principles from a Hamilton-Jacobi point of view.
	\smallskip

	\subsection{The novelties of this paper}
	
	We summarize the main novelties presented in this note.
	\begin{enumerate}[label=(\roman*)]
		\item We prove pathwise large deviations of slow components in general slow-fast systems via Hamilton-Jacobi theory. The large deviation principle can be reduced to well-posedness of a class of Hamilton-Jacobi-Bellman equations for which well-posedness was established in \cite{KrSc19}.
		\item We prove that the rate functions are of action-integral form, that means given by a time-integral over the so-called Lagrangian. Next to the standard characterization of the Lagrangian in terms of the dual to a principal eigenvalue, we establish a characterization in terms of a double-optimization. We use the latter characterization to derive averaging principles directly from the large-deviation results.
		\item As our main example, we treat density-flux large deviations of mean-field interacting particles on a finite state space coupled to fast drift-diffusion processes on a compact space. This example requires arguments that go beyond those currently available in the literature.
	\end{enumerate}
	Our large-deviation results also apply to small-diffusion processes coupled to fast jump processes. This solves a challenge pointed out in~\cite{BuDuGa18}, which is the fact that in slow-fast systems, classical results about comparison principles not applicable due to the Hamiltonians having poor regularity properties. In order to streamline this paper, we do not work out this example here; the key assumptions of our main theorems have been verified in \cite{KrSc19}.
	
	\smallskip
	
	The rest of this paper is organized as follows. 
	In Section \ref{section:general_setting}, we start with preliminaries and introduce our main context: the \textit{slow-fast} system. In Section~\ref{section:main_results} we give our two main results: Theorem \ref{theorem:LDP_general} and \ref{theorem:LDP_general_Lagrangian_two_represenatations}. We also give the application of these two results to the context of weakly interacting jump processes. In Section \ref{section:strategy_proof} we give the strategy of the proof of the main results. We work out the proof in Sections \ref{section:comparison} and \ref{section:variational_representation}. 
	In Section~\ref{SF:sec:mean-field-fast-diffusion}, we establish that our main results can indeed be applied in the context of weakly interacting jump processes.
	%%%%%%%%%%%%
	\section{General setting} \label{section:general_setting}
	%%%%%%%%%%%%
	In this section, we start with some preliminary definitions, after which we introduce general slow-fast systems as a  class of two-component Markov processes $(X_n(t),Z_n(t))$ taking values in $E_n \times F$. The first component~$X_n(t)$ corresponds to the slow process, the second component~$Z_n(t)$ to the fast process.  
	% 	\todo{this text should be joined with the other intros}
	%%%%%%%%%%%%
	\subsection{Preliminaries} \label{section:SF:preliminaries}
	%%%%%%%%%%%%
	For a Polish space $\cX$, we denote by $C(\cX)$ and $C_b(\cX)$ the spaces of continuous and bounded continuous functions respectively. If $\cX \subseteq \bR^d$ then we denote by $C_c^\infty(\cX)$ the space of smooth functions that vanish outside a compact set in $\cX$. We denote by $C_{cc}^\infty(\cX)$ the set of smooth functions that are constant outside of a compact set, and by $\cP(\cX)$ the space of probability measures on $\cX$. We equip $\cP(\cX)$ with the weak topology, that is, the one induced by convergence of integrals against bounded continuous functions.
	We denote $\bR^+ = [0,\infty)$. $D_\cX(\bR^+)$ is the Skorokhod space of trajectories that are right-continuous and have left limits. We equip this space with its usual Skorokhod topology, see \cite{EK86}.
	
	\smallskip
	
	We assume that $E$ is a closed convex subset of $\bR^d$ which is contained in the $\bR^d$ closure of its $\bR^d$ interior. This ensures that gradients of functions on $E$ are determined by their values on $E$. $E$ serves as the state-space of our slow process. We furthermore assume that $F$ is a compact metric space which serves as the state-space of the fast process.

	As a final technical definition, we introduce the space $\cA\cC(E)$ of absolutely continuous paths in $E$. A curve $\gamma: [0,T] \to E$ is absolutely continuous if there exists a function $g \in L^1([0,T],\bR^d)$ such that for $t \in [0,T]$ we have $\gamma(t) = \gamma(0) + \int_0^t g(s) \dd s$. We write $g = \dot{\gamma}$.\\
	A curve $\gamma: \bR^+ \to E$ is absolutely continuous, i.e. $\gamma \in \cA\cC(E)$, if the restriction to $[0,T]$ is absolutely continuous for every $T \geq 0$. 
	
	\smallskip
	
	We proceed with the main definitions of the theory of large deviations.

	\begin{definition}
		Let $\{X_n\}_{n \geq 1}$ be a sequence of random variables on a Polish space $\mathcal{X}$. Furthermore, consider a function $I : \mathcal{X} \rightarrow [0,\infty]$ and a sequence $\{r_n\}_{n \geq 1}$ of positive numbers such that $r_n \rightarrow \infty$. We say that
		\begin{itemize}
			\item  
			the function $I$ is a \textit{good rate-function} if the set $\{x \, | \, I(x) \leq c\}$ is compact for every $c \geq 0$.
			%
			%		\item 
			%		the sequence $\{X_n\}_{n\geq 1}$ is \textit{exponentially tight} at speed $r_n$ if, for every $a \geq 0$, there exists a compact set $K_a \subseteq \mathcal{X}$ such that $\limsup_n r_n^{-1} \log \, \PR[X_n \notin K_a] \leq - a$.
			%		%
			\item 
			the sequence $\{X_n\}_{n\geq 1}$ satisfies the \textit{large deviation principle} with speed $r_n$ and good rate-function $I$ if for every closed set $A \subseteq \mathcal{X}$, we have 
			\begin{equation*}
				\limsup_{n \rightarrow \infty} \, \frac{1}{r_n} \log \PR[X_n \in A] \leq - \inf_{x \in A} I(x),
			\end{equation*}
			and, for every open set $U \subseteq \mathcal{X}$, 
			\begin{equation*}
				\liminf_{n \rightarrow \infty} \, \frac{1}{r_n} \log \PR[X_n \in U] \geq - \inf_{x \in U} I(x).
			\end{equation*}
			%		, denoted by 
			%		\begin{equation*}
			%		\PR[X_n \approx a] \asymp e^{-r_n I(a)},
			%		\end{equation*}
		\end{itemize}
	\end{definition}
	
	Next, we introduce Markov processes via the martingale problem. Let $A : \cD(A) \subseteq C_b(\cX) \rightarrow C_b(\cX)$ be a linear operator.
	
	\begin{definition}
		Let $\mu \in \cP(\cX)$. We say that a measure $\PR \in \cP(D_\cX(\bR^+))$ solves \textit{the martingale problem} for $(A,\mu)$ if for all $f \in \cD(A)$ the process
		\begin{equation*}
			M_f(t) := f(X(t)) - f(X(0)) - \int_0^t Af(X(s)) \dd s
		\end{equation*}
		is a martingale with respect to the filtration $t \mapsto \cF_t := \left\{X(s) \, | \, s \leq t\right\}$, and if the projection of $\PR$ on the time $0$ coordinate equals $\mu$.
		
		We write that $\PR \in \cP(D_\cX(\bR^+))$ solves the martingale problem for $A$ if it solves the martingale problem for $(A,\mu)$ for some $\mu$. Finally, we say that the process $\{X(t)\}_{t \geq 0}$ on $D_\cX(\bR^+)$ solves the martingale problem for $A$ if its distribution solves the martingale problem.
		
		We say that the martingale problem is \textit{well-posed} if there exists a unique solution to the martingale problem for each starting measure.
	\end{definition}

	%%%%%%%%%%
	\subsection{Stochastic slow-fast systems}
	\label{subsection:slow_fast_systems}
	%%%%%%%%%%

	We proceed with introducing the general context of a slow-fast system. We start of by introducing the state-spaces for slow and fast components, after which we introduce the slow-fast system as the solution to a suitable martingale problem.
	
	\smallskip
	
	To focus only on the features that arise due to the coupling of slow and fast variables, we will assume that the fast process $Z_n(t)$ takes values in a compact Polish space $F$. This compactness assumption, as well as the fact that $F$ does not depend on $n$, can both be relaxed at the cost of more, but non-trivial, arguments~\cite{Gh18}. 
	
	\smallskip
	
	Regarding the state-spaces of the sequence of slow processes we allow for a changing sequence of spaces. This occurs for example in the context of interacting jump processes as in Section \ref{SF:sec:mean-field-fast-diffusion} on $\{1,\dots,q\}$.
	In such a context $E_n = \{1,\dots,q\}^n$, $E = \cP(\{1,\dots,q\} \subseteq \bR^q$  and $E_n$ naturally embeds into $E$ by the map $\eta_n : E_n \rightarrow E$ that takes the configuration $\vec{x} = (x_1,\dots,x_n) \in E_n$ to the empirical measure $\frac{1}{n} \sum_{i=1}^n \delta_{x_i}$.
	
	\smallskip
	
	More generally, we assume that the slow process $X_t^n$ takes values in Polish spaces $E_n$ such that  $\eta_n(E_n) \subseteq E \subseteq \mathbb{R}^d$, where $\eta_n : E_n \to E$ is a continuous embedding and $E$ is a Polish space as well. We assume that $E$ is contained in the $\mathbb{R}^d$-closure of its $\mathbb{R}^d$-interior, which ensures that gradients of functions on $E$ are determined by the values of the function in $E$. The setting of the state spaces is summarized in the following basic condition.

	\begin{condition}[Basic condition on the state spaces $E_n$ and $F$]
		\label{condition:compact_setting:state-spaces}
		The state space $F$ is a compact Polish space. The state spaces $E_n$ are Polish spaces that are asymptotically dense in $E \subseteq \mathbb{R}^d$ with respect to continuous embeddings $\eta_n : E_n \to E$; that means for any $x \in E$, there exist $x_n \in E_n$ such that $\eta_n(x_n) \to x$ as $n \to \infty$. Furthermore, we assume that for each compact $K \subseteq E$ the set $\eta_n^{-1}(K)$ is compact in $E_n$ and that there exists a compact set $\widehat{K} \subseteq E$ such that 
		\begin{equation*}
			K \subseteq \liminf_n \eta^{-1}_n(\widehat{K}).
		\end{equation*}
		The last condition means that for every compact	$K \subseteq E$ there is a compact set $\widehat{K} \subseteq E$ such that for all $x \in K$ there is an increasing map $k : \bN\rightarrow \bN$ and $x_{k(n)} \in \eta_{k(n)}^{-1}(\widehat{K})$ such that $\lim_n \eta_{k(n)}(x_{k(n)}) = x$.
	\end{condition}
	We will speak of a slow-fast system when the sequence of processes $t \mapsto (Y_n(t),Z_n(t))$ solve the martingale problem for a operators $A_n$ that decompose into a `slow' and a `fast' part whose timescale separation tends to infinity.
	
	\begin{definition}[Generator of a slow-fast system] \label{definition:slow_fast_system}
		We say that a sequence of linear operators $A_n : \mathcal{D}(A_n) \subseteq C_b(E_n \times F) \to C_b(E_n \times F)$ corresponds to a slow-fast system if $A_n$ is given by
		\begin{equation}\label{eq:setting:generator-slow-fast-system}
			A_n f(y,z) := A^\mathrm{slow}_{n,z} f(\cdot,z) (y) +
			r_n \cdot A^\mathrm{fast}_{n,y} f(y,\cdot)(z),
		\end{equation}
		where 
		%\todo{R: want to use $A_{n,y}^\mathrm{fast}$, as the spaces $E_n$ vary. So its natural to make it $n$ dependent. $y$, so that we can say $x = \eta_n(y)$}
		\begin{enumerate}[label = (\roman*)]
			\item $r_n$ is a sequence of positive real numbers such that $r_n \rightarrow \infty$.
			\item For each $z \in F$ and $n = 1,2,\dots$, there is a generator 
			\begin{equation*}
				A^\mathrm{slow}_{n,z} : \mathcal{D}(A^\mathrm{slow}_{n}) \subseteq C_b(E_n) \to C_b(E_n)
			\end{equation*}
			of an $E_n$-valued Markov process $Y^n_t$. The domain of $A^{\mathrm{slow}}_{n,z}$ is independent of $z$, and we denote it by $\mathcal{D}(A^\mathrm{slow}_n)$. For all $f\in \cD(A_n)$, we have $f(\cdot,z) \in \cD(A^\mathrm{slow}_n)$.
			\item For each $y \in E_n$, there is a generator
			\begin{equation*}
				A^\mathrm{fast}_{n,y} :  \mathcal{D}(A^\mathrm{fast}_{n,y}) \subseteq C(F) \to C(F)
			\end{equation*}
			of a Markov process on $F$. The domain of $A^\mathrm{fast}_{n,y}$ is independent of $n$ and $y$, and we denote it by $\mathcal{D}(A^\mathrm{fast})$. For all $f \in \cD(A_n)$, we have $f(y,\cdot) \in \mathcal{D}(A^\mathrm{fast})$.
		\end{enumerate}
	\end{definition}
	For a sequence of slow-fast systems constructed from operators $A_n$ defined as above, we make the following well-posedness assumption regarding solvability of the associated martingale problem.
	\begin{condition}[Well-posedness of martingale problem]
		\label{condition:compact_setting:well-posedness-martingale-problem}
		Consider a slow-fast system constructed from operators $A_n$ as in Definition \ref{definition:slow_fast_system}.	
		For each $n \in \mathbb{N}$ and each initial distribution $ \mu \in \mathcal{P}(E_n \times F)$, existence and uniqueness hold for the $(A_n,\mu)$-martingale problem on the Skorohod-space $D_{E_n \times F}[0,\infty)$. Denote the Markov process solving the martingale problem by $(Y_n(t),Z_n(t))$. The mapping $(y,z) \mapsto P^n_{y,z}$ of $E_n \times F$ into $\mathcal{P}(D_{E_n \times F}[0,\infty))$ is continuous with respect to the weak topology on $\mathcal{P}(D_{E_n \times F}[0,\infty))$, where $P^n_{y,z}$ is the distribution of the Markov process $(Y_n(t),Z_n(t))$ starting at $(y,z)$.
	\end{condition}

	%%%%%%%%%%%%%%
	\section{Main results} \label{section:main_results}
	%%%%%%%%%%%%%%

	We start off in Section \ref{sec:results:general-framework} with our two main results: path-space large deviations for general slow-fast systems, and an action-integral representation of the rate function. In addition to the conditions on the decomposition of the state-space and processes in a slow and fast component, we need various additional conditions that imply large deviations for both parts separately and weak-dependence on each other. We state and discuss these additional conditions in Section \ref{section:assumptions}. We give a main application of our general result in the context of weakly interacting jump processes in Section \ref{SF:sec:mean-field-fast-diffusion}.

	\subsection{A general framework for two time-scale Markov processes}
	\label{sec:results:general-framework}
	%%%%%%%%%%%%%%

	We consider a slow-fast system $t \mapsto (Y_n(t),Z_n(t))$ corresponding to the generators $A_n$ with decomposition
	\begin{equation*}
		A_n f(y,z) := A^\mathrm{slow}_{n,z} f(\cdot,z) (y) +
		r_n \cdot A^\mathrm{fast}_{n,y} f(y,\cdot)(z),
	\end{equation*}
	as in Conditions \ref{condition:compact_setting:state-spaces} and \ref{condition:compact_setting:well-posedness-martingale-problem}. Clearly, such a bare context is not sufficient to obtain large deviations for the slow-fast system. At the bare minimum, we need to impose additional assumptions that are sufficient to establish large deviations for the two components separately.
	
	We will thus make various assumptions on the limits $V_{x,p}(z)$ and $A_{x}^\mathrm{fast}$ of the operators $A_{n,z}^\mathrm{slow}$ and $A_{n,y}^\mathrm{fast}$. These assumptions will essentially imply:
	\begin{itemize}
		\item Path-space large deviations for the slow processes if the fast process is frozen.
		\item Donsker-Varadhan large deviations for the fast process if the slow process is frozen.
		\item `Regularity' of these two large deviation principles in variations of the frozen parameter.
	\end{itemize}
	Even though these assumptions seem to be minimal and natural, their discussion is quite lengthy. We therefore post-pone their discussion to Section \ref{section:assumptions} below. The following theorem is proven in Sections~\ref{section:strategy_proof} and~\ref{section:comparison}.
	%%%
	\begin{theorem}[Large deviations]\label{theorem:LDP_general}
		Let Condition \ref{condition:compact_setting:state-spaces} be satisfied and let~$(Y_n,Z_n)$ be the slow-fast system corresponding to the generators $A_n$ as in Condition \ref{condition:compact_setting:well-posedness-martingale-problem}. Denote $X_n = \eta_n(Y_n)$.
		Suppose that the large deviation principle holds for~$X_n(0)$ on~$E$ with speed~$r_n$ and good rate function~$\mathcal{J}_0$.
		In addition suppose Assumptions~\ref{assumption:compact_setting:convergence-of-nonlinear-generators}, \ref{assumption:compact_setting:convergence-of-nonlinear-generators-external}, \ref{assumption:compact_setting:eigen_value}, \ref{assumption:results:regularity_of_V} and ~\ref{assumption:results:regularity_I} are satisfied.  
		
		Then the large deviation principle holds with speed~$r_n$ for the process~$X_n$ on~$D_E(\mathbb{R}^+)$ with good rate function~$J$ given in \eqref{eqn:LDP_rate2}. 
	\end{theorem}
	
	The next theorem works out \eqref{eqn:LDP_rate2} in two forms under an additional assumption on the behaviour of $V_{x,p}$ on the boundary of $E$.

	\begin{theorem}[Action-integral representation]\label{theorem:LDP_general_Lagrangian_two_represenatations}
		In addition to the assumptions of Theorem~\ref{theorem:LDP_general}, suppose that Assumption~\ref{assumption:Hamiltonian_vector_field} is satisfied. Then the rate function~$J:D_E(\mathbb{R}^+)\to[0,\infty]$ can be written in action-integral form,
		\begin{equation*}
			J(\gamma) = \begin{cases}
				J_0(\gamma(0)) + \int_0^\infty \cL(\gamma(s),\dot{\gamma}(s)) \dd s & \text{if } \gamma \in \cA\cC, \\
				\infty & \text{otherwise},
			\end{cases}
		\end{equation*}
		where the Lagrangian~$\mathcal{L}$ admits the following two representations:
		\begin{enumerate}[label=(\roman*)]
			\item The map $v\mapsto \mathcal{L}(x,v)$ is the Legendre dual of the principal eigenvalue~$\mathcal{H}(x,p)$ of the operator $V_{x,p}+A^\mathrm{fast}_x$. In other words: $\mathcal{L}(x,v) = \sup_{p}\ip{p}{v} -\mathcal{H}(x,p)$,
			where the Hamiltonian~$\mathcal{H}(x,p)$ is given by
			\begin{equation}
				\label{eqn:variational_Hamiltonian}
				\cH(x,p) = \sup_{\pi \in \cP(F)} \left\{\int V_{x,p}(z) \, \pi(\dd z) - \cI(x,\pi) \right\},
			\end{equation}
			where $V_{x,p}(z)$ is the internal Hamiltonian from Assumption~\ref{assumption:compact_setting:convergence-of-nonlinear-generators}  and~$\mathcal{I}(x,\cdot):\mathcal{P}(F)\to[0,\infty]$ is the Donsker-Varadhan functional given by
			\begin{equation} \label{eqn:def_DV_functional}
				\cI(x,\pi) = - \inf_{\substack{\phi \in \cD(A^\mathrm{fast}) \\ \inf \phi > 0}} \int \frac{A_x^\mathrm{fast} \phi(z)}{\phi(z)} \, \pi(\dd z).
			\end{equation}
			\item The map~$\mathcal{L}$ is given by
			\begin{multline*}
				\cL(x,v) = \inf\left\{ \int \cL_z(x,w(z)) \, \nu(\dd z) + \cI(x,\nu) \, \middle| \, \nu \in \cP(F), \right. \\
				\left. w : E \rightarrow \bR^d \text{ $\nu$-integrable and } \int w(z) \, \nu(\dd z) = v  \right\},
			\end{multline*}
			where
			\begin{equation*}
				\cL_z(x,v) = \sup_p \ip{p}{v} - V_{x,p}(z).
			\end{equation*}
		\end{enumerate} 
	\end{theorem}
	
	\begin{remark}
		The assumption that $E$ is closed and convex is only used to establish the integral representation in Theorem \ref{theorem:LDP_general_Lagrangian_two_represenatations}. We can imagine that other methods to obtain this result in different contexts are available, see also the discussion following Assumption 2.17 in \cite{KrSc19}.
	\end{remark}

	%%%%%%%%%%%%
	\subsection{Assumptions for general theorems} \label{section:assumptions}
	%%%%%%%%%%%%
	Here we formulate the precise assumptions of the general large-deviation theorems given in Section~\ref{sec:results:general-framework}. Our proof is based on the connection between large deviations and Hamilton-Jacobi equations as first introduced by Feng and Kurtz. We explain this method as well as state the key results in this method in Section \ref{section:strategy_proof}. 
	
	\smallskip
	
	In this general framework, one has to check various assumptions for the specific models under consideration. Our contribution is to translate these assumptions to verifiable assumptions in the two-scale context. The assumptions naturally cluster in three groups, where each group corresponds to a main step in the proof:
	\begin{enumerate}
		\item Convergence of non-linear generators.
		\item Comparison principle of the limiting Hamilton-Jacobi-Bellman equation.
		\item Action-integral form of the rate function.
	\end{enumerate}

	\subsubsection{Assumptions for the convergence of non-linear generators}
	
	Our first two assumptions refer to the convergence of the internal and external generators. We state  them separately.
	
	\begin{assumption}[Convergence of internal non-linear generators]	\label{assumption:compact_setting:convergence-of-nonlinear-generators}
		Let $D_0$ be a linear space such that $C_c^\infty(E) \subseteq D_0 \subseteq C_b^1(E)$ and such that
		\begin{itemize}
			\item for any $n$, $z \in F$ and $f \in D_0$, $z \in F$ we have $e^{r_n f} \in \cD(A_{n,z}^\mathrm{slow})$ and
			\begin{equation*}
				\sup_{n}\sup_{x \in E_n,z \in F} \left|\frac{1}{r_n} e^{-r_nf(x,z)} A_{n,z}^\mathrm{slow} e^{r_nf(\cdot,z)(x)} \right| < \infty;
			\end{equation*}
			\item there exist continuous functions $V_{x,p} : F \to \mathbb{R}$, where $x \in E$  and $p \in \mathbb{R}^d$, such that for any $f \in D_0$ and all compact sets $K \subseteq E$ we have
			\begin{equation*}
				\sup_{ x \in \eta_n^{-1}(K), z \in F} \left|
				\frac{1}{r_n} e^{-r_nf(x,z)} A_{n,z}^\mathrm{slow} e^{r_nf(\cdot,z)(x)} - V_{\eta_n(x),\nabla f(\eta_n(x))}(z)
				\right| \to 0.
			\end{equation*}
		\end{itemize}
	\end{assumption}

	\begin{assumption}[Convergence of external non-linear generators]
		\label{assumption:compact_setting:convergence-of-nonlinear-generators-external}
		The external generators $A^\mathrm{fast}_{n,y}$ depend on $y \in E_n$ such that for any $x \in E$, $x_n \in E_n$ such that $\eta_n(x_n) \rightarrow x$ and $g_n \in \mathcal{D}(A^\mathrm{fast}) \subseteq C(F)$  with $g_n \to g$ uniformly on $F$, we have $A^\mathrm{fast}_{n,x_n}g_n \to A^\mathrm{fast}_x g$ uniformly on $F$.  For all $\phi \in \mathcal{D}(A^\mathrm{fast})$ we have $\sup_n \sup_{y \in E_n} \vn{e^{-\phi(z)} \left(A^\mathrm{fast}_{n,y} e^{-\phi}\right)(z)} < \infty$.
	\end{assumption}
	With the above two assumptions, we will obtain a limit operator~$H$ defined in terms of a graph~$H\subseteq C_b(E)\times C_b(E\times F)$. The precise definition of~$H$ is given further below in the proofs. 
	
	\subsubsection{Assumptions for the comparison principle of the limiting Hamilton-Jacobi-Bellman equation.}
	
	We proceed with three analytical assumptions. 
	
	\smallskip
	
	The first assumption is a mixing condition for the fast process and essentially implies that we have large deviations for the occupation time measures of the fast process when the slow process remains fixed. This mixing condition allows us to relate the Hamiltonian obtained from the limiting procedure to the Hamiltonian $\cH(x,p)$ of
	\eqref{eqn:variational_Hamiltonian} in Theorem \ref{theorem:LDP_general}.

	% 	In our companion paper \cite{KrSc19} we establish well-posedness of Hamilton-Jacobi-Bellman for a class of Hamiltonians that includes the Hamiltonian of Theorem \ref{theorem:LDP_general}. Our results are based on this comparison principle. We therefore include the two key assumptions of \cite{KrSc19} in which these assumptions are discussed at length.
	
	% 	As a key remark, however, the \mi{following assumption \sout{first Assumption}} essentially implies the large deviation principle for the slow component conditional on the stationarity of the fast component, as the continuity estimate is essentially a comparison principle in itself.

	\smallskip

	\begin{assumption}[Principal-eigenvalue problem] 
		\label{assumption:compact_setting:eigen_value}
		
		The operators $A^\mathrm{fast}_x$ satisfy the following.
		\begin{enumerate}[label=($\cE$\arabic*)]
			\item \label{item:assumption:PI:domain} For any $g > 0$ in $\cD(A^\mathrm{fast})$ and $\varepsilon \in (0,1)$, we have $g^{1-\varepsilon} \in \cD(A^\mathrm{fast})$.
			\item \label{item:assumption:PI:solvePI} For all $x \in E$ and any $p \in \mathbb{R}^d$, the principal-eigenvalue problem for the operator $V_{x,p} + A^\mathrm{fast}_x$ is well-posed. 
			
			That means, for every $\delta > 0$ there exists a strictly positive function $g \in \cD(A^\mathrm{slow})$  on $F$ such that
			\begin{equation*}
				\sup_{z \in F} \left| 	\left( V_{x,p}(z) + A^\mathrm{fast}_x \right) g(z) - \cH(x,p) g(z) \right| \leq \delta,
			\end{equation*}
			where $V_{x,p}(z)$ acts by multiplication.
		\end{enumerate}
	\end{assumption}
	
	\begin{remark}
		Part~\ref{item:assumption:PI:domain} 
		% is a mild assumption on the domain which is satisfied for our examples. More generally, it 
		is satisfied if the domain is closed under composition with smooth functions, and holds true for our example.
		% most interesting examples. 
		% Since $F$ is compact, 
		% and $\mathcal{D}(A^\mathrm{fast})\subseteq C(F)$, \ref{item:assumption:PI:domain} 
		% it
		% says that whenever a continuous function in the domain satisfies $C^{-1}\leq g \leq C$ for some constant~$C>0$, then $g^{1-\varepsilon}$ is in the domain as well. This typically follows as the uniform bounds on $g$ imply that $g^{1-\varepsilon}$ shares the same regularity.
		Part~\ref{item:assumption:PI:solvePI} is satisfied when there exists a positive function $g$ in the domain of $L_x$ such that $(V_{x,p}+L_x)g=\mathcal{H}(x,p)g$. In examples, this principal-eigenvalue problem is usually well-posed under standard regularity assumptions on coefficients, as a consequence of the Krein-Rutman theorem.
	\end{remark}

	% 	The assumption on the eigenvalue problem above establishes that the limiting operator obtained from Assumptions \ref{assumption:compact_setting:convergence-of-nonlinear-generators} and \ref{assumption:compact_setting:convergence-of-nonlinear-generators-external} can be related to the operator $\cH$. 
	% 	The following two assumptions are imposed on~$V$ and~$\mathcal{I}$ in order to verify the comparison principle for Hamilton-Jacobi-Bellman equations involving the Hamiltonian~$\mathcal{H}(x,p)$. 
	
	\smallskip
	
	%The formulation of the first assumption is based on a few concepts.	One of these concepts is the continuity estimate which is essentially a comparison principle for the Hamilton-Jacobi equations $f(x) - \lambda V_{x,\nabla f(x)}(z) = h(x)$ uniformly in $z$. We showed in \cite{KrSc19} that the continuity estimate, in combination with control on the Donsker-Varadhan rate function, allows to establish the comparison principle for the Hamilton-Jacobi-Bellman equation for $f - \lambda \cH f = h$. 
	
	% On the probabilistic side, the continuity estimate, being a `uniform'  comparison principle for the $V$ operators, implies (given sufficient tightness) path-space large deviation principles for the slow process when the fast process is frozen.
	The following two assumptions are imposed on~$V$ and~$\mathcal{I}$ in order to verify the comparison principle for Hamilton-Jacobi-Bellman equations involving the Hamiltonian~$\mathcal{H}(x,p)$.
	Roughly speaking, the continuity estimate implies path-space large deviation principles for the slow process when the fast process is frozen. We start introducing terminology that is motivated by the results in \cite{KrSc19}.
	%%%

	\begin{definition}[Penalization function]\label{def:results:good_penalization_function}
		We say that $\Psi : E^2 \rightarrow [0,\infty)$ is a \textit{ penalization function} if $\Psi \in C^1(E^2)$ and if $x = y$ if and only if $\Psi(x,y) = 0$.
	\end{definition}
	%%%
	\begin{definition}[Containment function]\label{def:results:compact-containment}
		We say that a function $\Upsilon : E \rightarrow [0,\infty]$ is a \textit{containment function} for $\Lambda$ if $\Upsilon \in C^1(E)$ and there is a constant $c_\Upsilon$ such that
		\begin{itemize}
			\item For every $c \geq 0$, the set $\{x \, | \, \Upsilon(x) \leq c\}$ is compact;
			\item We have $\sup_z \sup_x V_{x,\nabla \Upsilon(x)}(z) \leq c_\Upsilon$.
		\end{itemize}	
	\end{definition}
	%%%

	We proceed with the continuity estimate which we will later employ for $\cG(x,p,\pi) = \int V_{x,p}(z) \pi(\dd z)$.

	\begin{definition}[Continuity estimate] \label{def:results:continuity_estimate}
		Let  $\Psi$ be a penalization function and let $\cG: E \times \mathbb{R}^d\times\cP(F) \rightarrow \bR$, $(x,p,\pi)\mapsto \cG(x,p,\pi)$ be a function. Suppose that for each $\varepsilon > 0$, there is a sequence of positive real numbers $\alpha \rightarrow \infty$. For sake of readability, we suppress the dependence on $\varepsilon$ in our notation.
		
		Suppose that for each $\varepsilon$ and $\alpha$ we have variables $(x_{\varepsilon,\alpha},y_{\varepsilon,\alpha})$ in $E^2$ and measures $\pi_{\varepsilon,\alpha}$ in $\cP(F)$. We say that this collection is \textit{fundamental} for $\cG$ with respect to $\Psi$ if:
		\begin{enumerate}[label = (C\arabic*)]
			\item \label{item:def:continuity_estimate:1} For each $\varepsilon$, there are compact sets $K_\varepsilon \subseteq E$ and $\widehat{K}_\varepsilon\subseteq \cP(F)$ such that for all $\alpha$ we have $x_{\varepsilon,\alpha},y_{\varepsilon,\alpha} \in K_\varepsilon$ and $\pi_{\varepsilon,\alpha}\in\widehat{K}_\varepsilon$.
			\item \label{item:def:continuity_estimate:2} 
			For each $\varepsilon > 0$, we have $\lim_{\alpha \rightarrow \infty} \alpha \Psi(x_{\varepsilon,\alpha},y_{\varepsilon,\alpha}) = 0$. For any limit point $(x_\varepsilon,y_\varepsilon)$ of $(x_{\varepsilon,\alpha},y_{\varepsilon,\alpha})$, we have $\Psi(x_{\varepsilon},y_{\varepsilon}) = 0$.
			\item \label{item:def:continuity_estimate:3} We have for all $\varepsilon > 0$
			\begin{align} 
				& \sup_{\alpha} \cG\left(y_{\varepsilon,\alpha}, - \alpha (\nabla \Psi(x_{\varepsilon,\alpha},\cdot))(y_{\varepsilon,\alpha}),\pi_{\varepsilon,\alpha}\right) < \infty, \label{eqn:control_on_Gbasic_sup} \\
				& \inf_\alpha \cG\left(x_{\varepsilon,\alpha}, \alpha (\nabla \Psi(\cdot,y_{\varepsilon,\alpha}))(x_{\varepsilon,\alpha}),\pi_{\varepsilon,\alpha}\right) > - \infty. \label{eqn:control_on_Gbasic_inf} 	
			\end{align} \label{itemize:funamental_inequality_control_upper_bound}
			
			In other words, the operator $\cG$ evaluated in the proper momenta is eventually bounded from above and from below.
		\end{enumerate}
		We say that $\cG$ satisfies the \textit{continuity estimate} if for every fundamental collection of variables we have for each $\varepsilon > 0$ that
		\begin{multline}\label{equation:Xi_negative_liminf}
			\liminf_{\alpha \rightarrow \infty} \cG\left(x_{\varepsilon,\alpha}, \alpha (\nabla \Psi(\cdot,y_{\varepsilon,\alpha}))(x_{\varepsilon,\alpha}),\pi_{\varepsilon,\alpha}\right) \\
			- \cG\left(y_{\varepsilon,\alpha}, - \alpha (\nabla \Psi(x_{\varepsilon,\alpha},\cdot))(y_{\varepsilon,\alpha}),\pi_{\varepsilon,\alpha}\right) \leq 0.
		\end{multline}
		%%%
	\end{definition}

	%%%
	The continuity estimate essentially states that we have the comparison principle for the Hamilton-Jacobi equation for $(x,p) \mapsto \int V_{x,p}(z) \pi(\dd z)$ `uniformly' over $\cP(F)$. This can be rigorous if we also assume the existence of a containment function $\Upsilon$ and assume appropriate continuity and convexity of $V$. This is a key part of the following assumption.

	%The following assumption also contains two types of regularity in terms of the dependence of $V_{x,p}$ on the fast component $z$. This includes continuity in \ref{item:assumption:slow_regularity:continuity} and the property that $z \mapsto V_{x,p}(z)$ grows at a comparable rate for different $z \in F$ in \ref{item:assumption:slow_regularity:controlled_growth}.

	%%%
	\begin{assumption}\label{assumption:results:regularity_of_V}
		The function $V$ from Assumption \ref{assumption:compact_setting:convergence-of-nonlinear-generators} satisfies the following.
		\begin{enumerate}[label=($V$\arabic*)]
			%		\item For any $(x,p) \in E \times \mathbb{R}^d$, the map $\theta\mapsto \Lambda(x,p,\theta)$ is continuous on $\Theta$.
			\item \label{item:assumption:slow_regularity:continuity} For every $(x,p)$ we have $V_{x,p} \in C(F)$ and the map $(x,p) \mapsto V_{x,p}$ is continuous on $C(F)$ for the supremum norm.
			\item \label{item:assumption:slow_regularity:convexity} For any $x \in E$ and $z \in F$, we have that $p \mapsto V_{x,p}(z)$ is convex. Furthermore, we have $V_{x,0}(z) = 0$ for all $x,z$.
			\item \label{item:assumption:slow_regularity:compact_containment} There exists a continuous containment function $\Upsilon : E \to [0,\infty)$ in the sense of Definition~\ref{def:results:compact-containment}.
			\item \label{item:assumption:slow_regularity:controlled_growth} 
			For every compact set $K \subseteq E$, there exist constants $M, C_1, C_2 \geq 0$  such that for all $x \in K$, $p \in \mathbb{R}^d$ and all $z_1,z_2\in F$,
			\begin{align*}
				V_{x,p}(z_1) \leq  \max\left\{M,C_1 V_{x,p}(z_2) + C_2\right\}.
			\end{align*}
			\item \label{item:assumption:slow_regularity:continuity_estimate} The function $\Lambda(x,p,\pi) := \int V_{x,p}(z) \, \pi(\dd z)$ satisfies the continuity estimate.
		\end{enumerate} 
	\end{assumption}
	
	\begin{remark}
		Conditions~\ref{item:assumption:slow_regularity:continuity} and~\ref{item:assumption:slow_regularity:convexity} will follow from the convergence assumption on the slow generators. We state them nevertheless to clarify the connection to~\cite{KrSc19}.
	\end{remark}
	
	\begin{remark}
		All the results of this paper also hold when the present continuity estimate is replaced by the version of Appendix \ref{section:continuity_estimate_general}. This is occasionally helpful for complicated Hamiltonians.
	\end{remark}
	
	We also assume basic regularity properties for the Donsker-Varadhan rate function $\cI$.
	%%% (no indent)
	For a compact set~$K\subseteq E$ and a constant~$M\geq 0$, write
	\begin{equation} \label{eqn:def:sublevelsets_I}
		\Theta_{K,M}:= \bigcup_{x \in K} \left\{\pi\in\cP(F) \, \middle| \,  \mathcal{I}(x,\pi) \leq M \right\},
	\end{equation}
	and
	\begin{equation}
		\Omega_{K,M}:= \bigcap_{x \in K} \left\{\pi \in \cP(F) \, \middle| \,  \mathcal{I}(x,\pi) \leq M \right\}.
	\end{equation}

	\begin{assumption}\label{assumption:results:regularity_I}
		The functional $\mathcal{I}:E\times\cP(F) \to [0,\infty]$ in~\eqref{eqn:def_DV_functional} satisfies the following.
		\begin{enumerate}[label=($\mathcal{I}$\arabic*)]
			\item \label{item:assumption:I:lsc} The map $(x,\pi) \mapsto \mathcal{I}(x,\pi)$ is lower semi-continuous on $E \times \cP(F)$.
			\item \label{item:assumption:I:zero-measure} For any $x\in E$, there exists a measure $\pi_x^0 \in \cP(F)$ such that $\mathcal{I}(x,\pi_x^0) = 0$. 
			\item \label{item:assumption:I:compact-sublevelsets}  For any compact set $K \subseteq E$ and constant $M$ the set $\Theta_{K,M}$ is compact in $\cP(F)$.
			\item \label{item:assumption:I:finiteness} For each $x \in E$, compact subset $K \subseteq \cP(F)$, there is an open neighbourhood $U \subseteq E$ of $x$ and constants $M',C_1',C_2' \geq 0$ such that for all $y \in U$ and $\pi \in K$ we have
			\begin{equation*}
				\cI(y,\pi) \leq \max \left\{M', C_1'\cI(x,\pi) + C_2' \right\}.
			\end{equation*}
			\item \label{item:assumption:I:equi-cont} For every compact set $K \subseteq E$ and each $M \geq 0$ the collection of functions $\{\cI(\cdot,\pi)\}_{\pi \in \Omega_{K,M}}$ is equicontinuous. That is: for all $\varepsilon > 0$, there is a $\delta > 0$ such that for all $\pi \in \Omega_{K,M}$ and $x,y \in K$ such that $d(x,y) \leq \delta$ we have $|\mathcal{I}(x,\pi) - \mathcal{I}(y,\pi)| \leq \varepsilon$.
		\end{enumerate}
	\end{assumption}
	
	%%%
	These assumptions are always satisfied for continuous and bounded $\cI$, but are also satisfied for much more elaborate functionals, like the one appearing in Section \ref{SF:sec:mean-field-fast-diffusion}. Condition~\ref{item:assumption:I:lsc} follows if the maps~$x\mapsto A_x^\mathrm{fast}\phi$ are continuous as a function of~$E$ to~$C(F)$ equipped with the supremum norm. Conditions~\ref{item:assumption:I:zero-measure} and~\ref{item:assumption:I:compact-sublevelsets} are always satisfied by the compactness assumption on~$F$. Again, we state these conditions to make the connection to~\cite{KrSc19} as clear as possible.
	
	\subsubsection{Assumptions for the action-integral form of the rate function}
	
	% 	To also establish that the rate function can be expressed in Lagrangian form, we 
	We shall assume some regularity of the Hamiltonian flow close to the boundary. 
	% 	In \cite{KrSc19}, we showed that 
	% 	we can give
	% there is
	% 	a variational representation for the viscosity solution to the Hamilton-Jacobi-Bellman equation $f(x) - \lambda \cH(x, \nabla f(x)) = h(x)$.
	% 	This we will use as an input for the variational representation of the rate function. 
	% 	The key condition to do so in \cite{KrSc19} was an assumption on the subgradient of the non-smooth $\cH$. 
	The assumption that follows can be dropped if there is no boundary, e.g. if  $E = \bR^d$. 
	In our context, $\cH$ is obtained from the functions $V_{x,p}$ and $\cI$. We translate Assumption 2.17 of \cite{KrSc19} to one on $V$.
	% 	For~$\Phi : \bR^d \rightarrow (-\infty,\infty]$ convex, the subdifferential set of~$\Phi$ is defined by
	%%%
	\begin{definition} \label{definition:tangent_cone}
		The tangent cone (sometimes also called \textit{Bouligand cotingent cone}) to $E$ in $\bR^d$ at $x$ is
		\begin{equation*}
			T_E(x) := \left\{z \in \bR^d \, \middle| \, \liminf_{\lambda \downarrow 0} \frac{d(y + \lambda z, E)}{\lambda} = 0\right\}.
		\end{equation*}
	\end{definition}
	%%% 
	Let $\Phi : \bR^d \rightarrow (-\infty,\infty]$ be convex. 
	Then the subdifferential set of $\Phi$ is defined by
	\begin{equation} \label{eqn:subdifferential}
		\partial_p \Phi(p_0) := \left\{
		\xi \in \mathbb{R}^d \,:\, \Phi(p) \geq \Phi(p_0) + \xi \cdot (p-p_0) \quad (\forall p \in \mathbb{R}^d)
		\right\}.
	\end{equation}
	\begin{assumption} \label{assumption:Hamiltonian_vector_field}
		The set~$E$ is closed and convex. The map $V : E \times \bR^d \times F \rightarrow \bR$ of Assumption \ref{assumption:compact_setting:convergence-of-nonlinear-generators} is such that $\partial_p V_{x,p_0}(z) \subseteq T_E(x)$ for all $p_0$, $x$ and $z$.
	\end{assumption}

	In~\cite{KrSc19}, this assumption is made on $\Lambda(x,p,\pi) = \int V_{x,p}(z) \pi(\dd z)$ instead of the integrand~$V_{x,p}(z)$. This property bootstraps from $V$ to $\Lambda$ as we will see in Proposition \ref{proposition:variational_resolvent}. 
	
	%%%%%%%%%%%%%

	\section{Strategy of the proof---the Hamilton-Jacobi approach to large deviations} \label{section:strategy_proof}
	
	A key role in the proof of our large deviation theorems is played by Hamilton-Jacobi (-Bellman) equations. This connection was first established by \cite{FK06} and reproved with new arguments in \cite{Kr19,Kr20}. An outline of the key steps in this argument is given in Section \ref{section:outline_of_proof}.
	
	A key role is played by viscosity solutions to certain Hamilton-Jacobi equations and their convergence. We introduce these concepts first.

	\subsection{Preliminaries}

	We next introduce viscosity solutions for the Hamilton-Jacobi equation with Hamiltonians like $\mathcal{H}(x,p)$ of our introduction. The notion of viscosity solutions is built up out of the notion of a sub- and supersolutions. For later flexibility, we will introduce two Hamilton-Jacobi equations instead of one and define sub- and supersolutions for the two equations respectively. These definitions are fairly technical and can be skipped until the moment they are needed.
	
	Let $C_u(E)$ be the space of continuous functions that have an upper bound, let $C_l(E)$ be the space of continuous functions with a lower bound.

	\begin{definition}[Viscosity solutions and comparison principle] \label{definition:viscosity_solutions}
		Let $B_1 \subseteq C_l(E) \times C(E \times F)$ and $B_2 \subseteq C_u(E) \times C(E \times F)$ be two operators, $\lambda > 0$ and $h_1,h_2 \in C_b(E)$. Consider the Hamilton-Jacobi equations
		\begin{align}
			f - \lambda B_1 f & = h_1, \label{eqn:differential_equation1}  \\
			f - \lambda B_2 f & = h_2. \label{eqn:differential_equation2} 
		\end{align}
		We say that $u$ is a \textit{(viscosity) subsolution} of equation \eqref{eqn:differential_equation1} if $u$ is bounded, upper semi-continuous and if, for all $(f,g) \in B_1$ there exists a sequence $(x_n,z_n) \in E\times F$ such that
		\begin{gather*}
			\lim_{n \uparrow \infty} u(x_n) - f(x_n)  = \sup_x u(x) - f(x), \\
			\lim_{n \uparrow \infty} u(x_n) - \lambda g(x_n,z_n) - h_1(x_n) \leq 0.
		\end{gather*}
		We say that $v$ is a \textit{(viscosity) supersolution} of equation \eqref{eqn:differential_equation2} if $v$ is bounded, lower semi-continuous and if, for every $(f,g) \in B_2$ there exists a sequence $(x_n,z_n) \in E\times F$ such that
		\begin{gather*}
			\lim_{n \uparrow \infty} v(x_n) - f(x_n)  = \inf_x v(x) - f(x), \\
			\lim_{n \uparrow \infty} v(x_n) - \lambda g(x_n,z_n) - h_2(x_n) \geq 0.
		\end{gather*}
		We say that $u$ is a \textit{(viscosity) solution} of equations \eqref{eqn:differential_equation1} and \eqref{eqn:differential_equation2} if it is a subsolution to \eqref{eqn:differential_equation1} and a supersolution to \eqref{eqn:differential_equation2}.
		
		We say that \eqref{eqn:differential_equation1} and \eqref{eqn:differential_equation2} satisfies the \textit{comparison principle} if for every subsolution $u$ to \eqref{eqn:differential_equation1} and supersolution $v$ to \eqref{eqn:differential_equation2}, we have $\sup_x u(x) - v(x) \leq \sup_x h_1(x) - h_2(x)$.
	\end{definition}

	\begin{remark}
		We recover the usual definition of viscosity solutions in terms of an operator $B$ on $C_b(E)$ if $B:= B_1 = B_2$ and if the operator $B$ is single-valued and its images do not depend on $F$. In this case we write $Bf = g$ if and only if $(f,g) \in B$.
	\end{remark}

	\begin{remark} \label{remark:existence of optimizers}
		Consider the context of the previous remark. Consider the definition of subsolutions. Suppose that the testfunction $f \in \cD(B)$ has compact sublevel sets, then instead of working with a sequence $x_n$, there exists $x_0  \in E$ such that
		\begin{gather*}
			u(x_0) - f(x_0)  = \sup_x u(x) - f(x), \\
			u(x_0) - \lambda B f(x_0) - h(x_0) \leq 0.
		\end{gather*}
		A similar simplification holds in the case of supersolutions.
	\end{remark}

	We next turn to the convergence of a sequence of functions on different spaces and the derived concept of an extended limit of operators.
	
	\begin{definition}
		Let $f_n \in C_b(E_n \times F)$ and $f \in C_b(E \times F)$. We say that $\LIM f_n = f$ if 
		\begin{itemize}
			\item $\sup_n \vn{f_n} < \infty$,
			\item for all compact $K \subseteq E$, we have
			\begin{equation*}
				\lim_{n \rightarrow \infty} \sup_{y \in \eta_n^{-1}(K) \times F} \left|f_n(y,z) - f(\eta_n(y),z) \right| = 0.
			\end{equation*}
		\end{itemize}
	\end{definition}
	
	\begin{definition}
		Let $B_n \subseteq C_b(E_n \times F) \times C_b(E_n\times F)$. Define $ex-\LIM B_n$ as the set
		\begin{multline*}
			ex-\LIM B_n \\
			= \left\{(f,g) \in C_b(E \times F)^2 \, \middle| \, \exists \, (f_n,g_n) \in B_n: \, \LIM f_n = f, \LIM g_n = g \right\}.
		\end{multline*}
	\end{definition}
	
	\begin{definition}
		We say that a sequence of functions $f_n \in C_b(E)$ converges strictly to $f \in C_b(E)$ if $\sup_n \vn{f_n} < \infty$ and if $f_n$ converges to $f$ uniformly on compacts. See \cite{Se72} for a topological treatment of the strict topology.
	\end{definition}

	\subsection{Outline of the proof} 
	\label{section:outline_of_proof}

	The framework of results below can either be obtained via \cite[Theorem 7.18]{FK06} or via \cite{Kr20}. The key result that we will use is Theorem 7.10 of \cite{Kr20}, which is based on the following argument.

	\begin{itemize}
		\item Given exponential tightness of the processes $X_n$, it suffices to establish the large deviations of the finite dimensional distributions.
		\item Using Brycs theorem and the Markov property, large deviations for the finite dimensional distributions follow from large deviations at time $0$ and the convergence $V_n(t) \rightarrow V(t)$ of the conditional generating functions 
		\begin{equation*}
			V_n(t)f(y) := \frac{1}{r_n} \log \bE\left[e^{r_n f(Y_n(t))} \, \middle| \, Y_n(0) = y\right].
		\end{equation*}
		\item The generating functions $V_n(t)$ form a non-linear operator semigroup. Following classical theory of semigroups, the convergence of convergence of these semigroups follows from the convergence of their non-linear generators 
		\begin{gather*}
			\cD(H_n) := \left\{f \in C_b(E_n \times F) \, \middle| \, e^{r_nnf} \in \cD(A_n) \right\}, \\
			Hf(x) = \frac{1}{r_n} e^{-r_n f} A_n e^{r_n f}
		\end{gather*}
		to some operator $H$ of which it needs to be shown that it generates a semigroup $V(t)$.
		\item As $H$ is non-linear, classical methods to show that $H$ generates a semigroup fail. We thus resort to viscosity methods. The sufficient condition under which we have this property is the comparison principle for the Hamilton-Jacobi equation in terms of $H$.
	\end{itemize}
	
	To make all of this rigorous, we additionally introduce the resolvents $"R_n(\lambda) = (\bONE - \lambda H_n)^{-1}"$ of the operators $H_n$:
	\begin{multline*}
		R_n(\lambda)h(x) := \\
		\sup_{\bQ \in \cP(D_{E_n \times F}(\bR^+))} \left\{ \int_0^\infty \lambda^{-1} e^{-\lambda^{-1}t} \left( \int h(X(t)) \bQ(\dd X) - \frac{1}{r_n} S_t(\bQ \, | \, \PR_x ) \right) \dd t\right\}. 
	\end{multline*}

	Before giving the key result of \cite{Kr20}, we further give a weakened exponential tightness property.
	
	\begin{definition}
		Consider the context of Conditions \ref{condition:compact_setting:state-spaces} and \ref{condition:compact_setting:well-posedness-martingale-problem}. We say that the processes $(Y_n(t),Z_n(t))$ satisfy the \textit{exponential compact containment condition} at speed $r_n$ if for each compact set $K \subseteq E$, $T >0$ and $a > 0$ there is a compact set $\widehat{K} = \widehat{K}(K,T,a) \subseteq E$ such that 
		\begin{equation*}
			\limsup_{n \rightarrow \infty} \sup_{(y,z) \in \eta_n^{-1}(K) \times F} \frac{1}{r_n} \log P_{y,z}\left[Y_n(t) \notin \eta_n^{-1}(\widehat{K}) \text{ for some } t \in [0,T] \right] \leq - a.
		\end{equation*}	
	\end{definition}

	\begin{theorem}[Adaptation of Theorem 7.10 of \cite{Kr20} to our context] \label{theorem:abstract_LDP}
		Suppose that we are in the setting of Conditions \ref{condition:compact_setting:state-spaces} and \ref{condition:compact_setting:well-posedness-martingale-problem} and that the exponential compact containment condition holds.
		
		Denote $X_n = \eta_n(Y_n)$. Suppose that
		\begin{enumerate}
			\item \label{item:LDP_abstract_initialLDP} The large deviation principle holds for $X_n(0) = \eta_n(Y_n(0))$ with speed $r_n$ and good rate function $J_0$.
			\item \label{item:LDP_abstract_exptight} The processes $X_n = \eta_n(Y_n)$ are exponentially tight on $D_E(\bR^+)$ with speed $r_n$.
			\item \label{item:LDP_abstract_convergenceHn} There is an operator $H \subseteq C_b(E) \times C_b(E \times F)$ such that $H \subseteq ex-\LIM H_n$.
			\item \label{item:LDP_abstract_comparison} For all $h \in C_b(E)$ and $\lambda > 0$ the comparison principle holds for $f - \lambda Hf = h$.
		\end{enumerate}
		Then there are two families of operators $R(\lambda) : C_b(E) \rightarrow C_b(E)$, $\lambda > 0$ and $V(t) : C_b(E) \rightarrow C_b(E)$, $t \geq 0$, such that 
		\begin{itemize}
			\item There is a sequentially strictly dense set $D \subseteq C_b(E)$ such that for each $t > 0$ and $f \in D$, we have
			\begin{equation} \label{eqn:convergence_resolvent_to_semigroup_abstract_LDP}
				\lim_{m \rightarrow \infty} \vn{R\left(\frac{t}{m}\right)^m f  - V(t)f  } = 0.
			\end{equation}
			\item If $\lambda > 0$ and $\LIM h_n = h$, then $\LIM R_n(\lambda) h_n = R(\lambda) h$;
			\item For $\lambda > 0$ and $h \in C_b(E)$, the function $R(\lambda)h$ is the unique function that is a viscosity solution to $f - \lambda H f = h$;
			\item If $\LIM f_n = f$ and $t_n \rightarrow t$ we have $\LIM V_n(t_n)f_n = V(t)f$.
		\end{itemize}
		In addition, the processes $X_n = \eta_n(Y_n)$ satisfy a large deviation principle on $D_E(\bR^+)$ with speed $r_n$ and rate function
		\begin{equation} \label{eqn:LDP_rate2}
			J(\gamma) = J_0(\gamma(0)) + \sup_{k \geq 1} \sup_{\substack{0 = t_0 < t_1 < \dots, t_k \\ t_i \in \Delta_\gamma^c}} \sum_{i=1}^{k} J_{t_i - t_{i-1}}(\gamma(t_i) \, | \, \gamma(t_{i-1})).
		\end{equation}
		Here $\Delta_\gamma^c$ is the set of continuity points of $\gamma$. The conditional rate functions $I_t$ are given by
		\begin{equation*}
			J_t(y \, | \, x) = \sup_{f \in C_b(E)} \left\{f(y) - V(t)f(x) \right\}.
		\end{equation*}
	\end{theorem}
	On the basis of this abstract result, we derive our main result.
	\begin{proof}[Proof of Theorem \ref{theorem:LDP_general}]
		To apply Theorem \ref{theorem:abstract_LDP}, we have to verify \ref{item:LDP_abstract_initialLDP} to \ref{item:LDP_abstract_comparison}. Assumption \ref{item:LDP_abstract_initialLDP}, as it is an initial condition, will be assumed from the outset. We verify \ref{item:LDP_abstract_exptight} in Proposition \ref{proposition:exponential_compact_containment} and \ref{item:LDP_abstract_convergenceHn} in Proposition \ref{proposition:convergence_of_Hamiltonians} below.
		Assumption~\ref{item:LDP_abstract_comparison} will be established in Theorem \ref{theorem:comparison}. 
	\end{proof}
	
	\subsection{A limiting operator} \label{section:limiting_operator}
	
	Our first goal is to establish that there is some operator $H$ such that $H \subseteq ex-\LIM H_n$. Due to Condition \ref{condition:compact_setting:well-posedness-martingale-problem} and Assumptions \ref{assumption:compact_setting:convergence-of-nonlinear-generators} and \ref{assumption:compact_setting:convergence-of-nonlinear-generators-external} there is a clear candidate for $H$.
	
	\begin{definition}
		The operator $H \subseteq C_b(E) \times C_b(E \times F)$ with $\cD(H) = C_{cc}^\infty(E)$ is multi-valued. For $f \in C_{cc}^\infty(E)$, $x \in E$ and $\phi$ such that $e^\phi \in \cD(A^\mathrm{fast})$  set 
		\begin{equation*}
			H_{f,\phi}(x,z) := V_{x,\nabla f(x)}(z) + e^{-\phi(z)} A^\mathrm{fast}_x e^{\phi}(z).
		\end{equation*} 
		The operator $H$ is given by
		\begin{equation*}
			H := \left\{(f,H_{f,\phi}) \, \middle| \, f \in C_{cc}^\infty(E), \phi: e^{\phi} \in \cD(A^\mathrm{fast}) \right\}.
		\end{equation*}
		
	\end{definition}

	\begin{proposition} \label{proposition:convergence_of_Hamiltonians}
		For all $(f,g) \in H$ there are $f_n \in \cD(H_n)$ such that $\LIM f_n = f$ and $\LIM H_n f_n = g$.
	\end{proposition}
	
	\begin{proof}
		Fix an arbitrary function $f \in D_0$ and $\phi \in \cD(A^\mathrm{fast})$. Set $f_n(x,z) := f(\eta_n(x)) + r_n^{-1} \phi(z)$. Note that $f_n \in \cD(H_n)$. We will prove that $\LIM f_n = f$ and $\LIM H_n f_n = H_{f,\phi}$, which will establish the claim.
		
		\smallskip
		
		As $f$ and $\phi$ are bounded, it follows that $\vn{f_n - f} = r_n^{-1}\vn{\phi} \rightarrow 0$, which implies $\LIM f_n = f$. By Condition \ref{condition:compact_setting:well-posedness-martingale-problem}, the images $H_n f_n$ are given by
		\begin{align*}
			H_n f_n(x,z)
			=
			\frac{1}{r_n} e^{-r_nf(x)} A_{n,z}^\mathrm{slow} e^{r_nf(x)} + e^{-\phi(z)} \left(A^{\mathrm{fast}}_x e^{\phi}\right)(z).
		\end{align*}
		Thus, the result follows by Assumptions \ref{assumption:compact_setting:convergence-of-nonlinear-generators} and \ref{assumption:compact_setting:convergence-of-nonlinear-generators-external}.
	\end{proof}

	\subsection{Exponential tightness}
	
	To establish exponential tightness, we first note that by \cite[Corollary 4.19]{FK06} or \cite[Proposition 7.12]{Kr19b} it suffices in our context to establish the exponential compact containment condition. This is the content of the next proposition.
	\begin{proposition} \label{proposition:exponential_compact_containment}
		For each compact set $K \subseteq E$, $T >0$ and $a > 0$ there is a compact set $\widehat{K} = \widehat{K}(K,T,a) \subseteq E$ such that 
		\begin{equation*}
			\limsup_{n \rightarrow \infty} \sup_{(y,z) \in \eta_n^{-1}(K) \times F} \frac{1}{r_n} \log P_{y,z}\left[Y_n(t) \notin \eta_n^{-1}(\widehat{K}) \text{ for some } t \in [0,T] \right] \leq - a.
		\end{equation*}	
	\end{proposition}
	
	\begin{proof}
		By Assumption \ref{assumption:results:regularity_of_V} \ref{item:assumption:slow_regularity:compact_containment} we have $\sup_{x,z} V_{x,\nabla \Upsilon(x)}(z) \leq c_\Upsilon$. Choose $\beta > 0$ such that $Tc_\Upsilon + 1 - \beta \leq -a$. As $\Upsilon$ is continuous, there is some $c$ such that
		\begin{equation*}
			K \subseteq \left\{x \, \middle| \, \Upsilon(x) \leq c \right\}
		\end{equation*}
		Next, set $G := \left\{x \, \middle| \, \Upsilon(x) < c + \beta \right\}$ and note that $G$ is open. Let $\widehat{K}$ be the closure of $G$. Note that $\widehat{K}$ is compact.
		
		Let $f(x) := \iota \circ \Upsilon $ where $\iota$ is some smooth increasing function such that
		\begin{equation*}
			\iota(r) = \begin{cases}
				r & \text{if } r \leq \beta +c, \\
				\beta+ c + 1 & \text{if } r \geq \beta + c + 2.
			\end{cases}
		\end{equation*}
		It follows that $\iota \circ \Upsilon = \Upsilon$ on $\widehat{K}$ and is constant outside of a compact set. Set $f_n = f \circ \eta_n$, $g_n := H_n f_n$ and $g := \LIM g_n$ (which exists due to Assumption \ref{assumption:compact_setting:convergence-of-nonlinear-generators}). Note that $g(x,z) = V_{x,\nabla \Upsilon(x)}(z)$ if $x \in \widehat{K}$. Therefore, we have $\sup_{x \in \widehat{K}, z \in F} g(x,z) \leq c_\Upsilon$.

		Let $\tau$ be the stopping time $\tau := \inf \left\{t \geq 0 \, \middle| \, Y_n(t) \notin \eta_n^{-1}(G) \right\}$ and let
		\begin{equation*}
			M_n(t) := \exp \left\{r_n \left( f_n(Y_n(t)) - f_n(Y_n(t)) - \int_0^t g_n(Y_n(t),Z_n(t)) \dd s \right) \right\}.
		\end{equation*}
		By construction $M_n$ is a martingale, and by the optional stopping theorem $t \mapsto M_n(t \wedge \tau)$ is a martingale also. We obtain
		\begin{align*}
			& \PR_{y,z}\left[Y_n(t) \notin \widehat{K} \text{ for some } t \in [0,T]\right] \\
			& \leq \PR_{y,z}\left[Y_n(t) \notin \eta_n^{-1}(G) \times F \text{ for some } t \in [0,T]\right] \\
			& = \bE_{y,z}\left[\bONE_{\{Y_n(t) \notin \eta_n^{-1}(G) \text{ for some } t \in [0,T]\}} M_n(t \wedge \tau) M_n(t\wedge \tau)^{-1} \right] \\
			& \leq \exp\left\{- r_n \left(\inf_{y_1 \in G^c} \Upsilon(\eta_n(y_1)) - \Upsilon(\eta_n(y))   \right. \right. \\
			& \hspace{4cm} \left. \left. - T \sup_{y_2 \in \eta_n^{-1}(G), z_2 \in F} g_n(y_2,z_2) \right) \right\} \\
			& \hspace{2.5cm} \times  \bE_{y,z}\left[\bONE_{\{Y_n(t) \notin \eta_n^{-1}(G) \text{ for some } t \in [0,T]\}} M_n(t \wedge \tau) \right].
		\end{align*}
		As $\LIM f_n = f$ and $\LIM g_n = g$, we obtain that the term in the exponential is bounded by $ r_n\left(c_\Upsilon T - \beta \right) \leq -r_n a$ for sufficiently large $n$. The final expectation is bounded by $1$ due to the martingale property of $M_n(t \wedge \tau)$.
		
		We conclude that
		\begin{equation*}
			\limsup_n \sup_{y \in \eta_n^{-1}(K), z \in F} \frac{1}{r_n} \log \PR_{y,z}\left[Y_n(t) \notin \eta^{-1}(\widehat{K}) \text{ for some } t \in [0,T]\right] \leq -a.
		\end{equation*}
		
	\end{proof}

	\section{Proof of Comparison Principle via a framework of Hamiltonians} \label{section:comparison}
	In this section, we establish the comparison principle for the Hamilton-Jacobi equation $f - \lambda Hf = h$ for $H$ that was introduced in Section \ref{section:limiting_operator} above. 
	To do so, we will relate solutions to the Hamilton-Jacobi equation for $H$ to solutions of a related Hamilton-Jacobi-Bellman equation in terms of the operator $\cH$ of \eqref{eqn:variational_Hamiltonian}. That $\cH$ is continuous was established in \cite[Appendix A]{KrSc19} on the basis of Assumptions \ref{assumption:results:regularity_of_V} and \ref{assumption:results:regularity_I}.

	\begin{definition} \label{definition_effectiveH}
		The operator $\bfH \subseteq C_b^1(E) \times C_b(E)$ has domain $\cD(\bfH) = C_{cc}^\infty(E)$ and satisfies $\bfH f(x) = \cH(x, \dd f(x))$, where $\cH$ is the map 
		\begin{equation} \label{eqn:def_variational_Hamiltonian}
			\cH(x,p) = \sup_{\pi \in \cP(F)} \left\{\int V_{x,p}(z) \pi(\dd z) - \cI(x,\pi) \right\}
		\end{equation}
		that was introduced in \eqref{eqn:variational_Hamiltonian}
		with $\cI$ as in \eqref{eqn:def_DV_functional}. We also repeat it for completeness:
		\begin{equation*} 
			\cI(x,\pi) = - \inf_{\substack{ g \in \mathcal{D}(A^\mathrm{fast}) \\ \inf   g > 0 }} \int \frac{A^\mathrm{fast}_x g(z)}{g(z)} \pi(\dd z).
		\end{equation*}
	\end{definition}
	The Hamilton-Jacobi-Bellman equation $f - \lambda \bfH f = h$ and two related equations $f - \lambda H_\dagger f = h$ and $f - \lambda H_\ddagger f = h$ were studied in more general form  in the accompanying paper \cite{KrSc19}. There we work with a general (non-compact) control  space $\Theta$ instead of $\cP(F)$. In addition, we allow for more general 'internal Hamiltonian' $\Lambda$ and cost function $\cI$.
	\smallskip
	
	We prove the comparison principle for the Hamilton-Jacobi equation in terms of $H$ by relating it to a set of Hamilton-Jacobi equations with Hamiltonians $H_\dagger,H_\ddagger$ constructed from $\cH$. The comparison principle for the Hamilton-Jacobi equations in terms of $H_\dagger,H_\dagger$ was established in \cite{KrSc19}. Effectively, \cite{KrSc19} establishes the blue box and the two arrows on the right of Figure \ref{figure:Hamiltonians} under a generalization of Assumptions~\ref{assumption:results:regularity_of_V} and \ref{assumption:results:regularity_I}.
	Below, we complete the figure by proving the left-hand side of the diagram.
	
	\begin{figure}[h!]
		\begin{center}
			\begin{tikzpicture}
				\matrix (m) [matrix of math nodes,row sep=1em,column sep=4em,minimum width=2em]
				{
					{ } & H_1 &[7mm] H_\dagger &[5mm] { } \\
					H & { } & { } & \bfH \\
					{ }  & H_2 & H_\ddagger & { } \\};
				\path[-stealth]
				(m-2-1) edge node [above] {sub} (m-1-2)
				(m-2-1) edge node [below] {super \qquad { }} (m-3-2)
				(m-1-2) edge node [above] {sub \qquad { }} (m-1-3)
				(m-3-2) edge node [below] {super \qquad { }} (m-3-3)
				(m-2-4) edge node [above] {\qquad sub} (m-1-3)
				(m-2-4) edge node [below] {\qquad super} (m-3-3);
				
				\begin{pgfonlayer}{background}
					\node at (m-2-3) [rectangle,draw=blue!50,fill=blue!20,rounded corners, minimum width=1cm, minimum height=2.5cm]  {comparison};
				\end{pgfonlayer}
			\end{tikzpicture}
		\end{center}
		\caption{An arrow connecting an operator $A$ with operator $B$ with subscript 'sub' means that viscosity subsolutions of $f - \lambda A f = h$ are also viscosity subsolutions of $f - \lambda B f = h$. Similarly for arrows with a subscript 'super'. The box around the operators $H_\dagger$ and $H_\ddagger$ indicates that the comparison principle holds for subsolutions of $f - \lambda H_\dagger f = h$ and supersolutions of $f - \lambda H_\ddagger f = h$.}
		\label{SF:fig:CP-diagram-in-proof-of-CP}
	\end{figure}
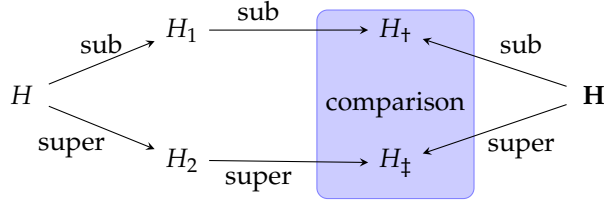

	We now prove the comparison principle for $f - \lambda Hf = h$ based on the results summarized in Figure~\ref{SF:fig:CP-diagram-in-proof-of-CP}. 
	\begin{theorem} \label{theorem:comparison}
		Let $h_1,h_2 \in C_b(E)$ and $\lambda >0$. Let $u$ be any subsolution to $f - \lambda H f = h_1$ and let $v$ be any supersolution to $f - \lambda H f = h_2$. Then we have that 
		\begin{equation*}
			\sup_x u(x) - v(x) \leq \sup_x h_1(x) - h_2(x).
		\end{equation*}
	\end{theorem}

	\begin{proof}
		Fix~$h_1,h_2\in C_b(E)$ and~$\lambda>0$. Let~$u$ be a viscosity subsolution and~$v$ be a viscosity supersolution to~$(1-\lambda H)f=h$. By Figure~\ref{SF:fig:CP-diagram-in-proof-of-CP}, the function~$u$ is a viscosity subsolution to~$(1-\lambda H_\dagger)f=h$ and~$v$ is a viscosity supersolution to~$(1-\lambda H_\ddagger)f=h$. Hence by the comparison principle for $H_\dagger,H_\dagger$ established in Theorem \ref{theorem:comparison_from_otherpaper} below,~$\sup_x u(x) - uv(x) \leq \sup_x h_1(x) - h_2(x)$, which finishes the proof.
	\end{proof}
	
	The rest of this section is devoted to establishing Figure~\ref{SF:fig:CP-diagram-in-proof-of-CP}. 
	
	\subsection{Definition of auxiliary  operators} \label{subsection:definition_of_Hamiltonians}

	We introduce the operators $H_\dagger,H_\ddagger$ and $H_1,H_2$. In both cases, the new Hamiltonians will serve as natural upper and lower bounds for $\bfH$ and $H$ respectively. These new Hamiltonians are defined in terms of the containment function $\Upsilon$, which allows us to restrict our analysis to compact sets. Recall Assumption~\ref{item:assumption:slow_regularity:compact_containment} and the constant $C_\Upsilon := \sup_{x,z} V_{x,\nabla \Upsilon(x)}(z)$ therein.

	Denote by $C_l^\infty(E)$ the set of smooth functions on $E$ that have a lower bound and by $C_u^\infty(E)$ the set of smooth functions on $E$ that have an upper bound.

	\begin{definition}
		\begin{itemize}
			\item For $f \in C_l^\infty(E)$  and $\varepsilon \in (0,1)$  set 
			\begin{gather*}
				f^\varepsilon_\dagger := (1-\varepsilon) f + \varepsilon \Upsilon, \\
				H_{\dagger,f}^\varepsilon(x) := (1-\varepsilon) \bfH f(x) + \varepsilon C_\Upsilon,
			\end{gather*}
			and set
			\begin{equation*}
				H_\dagger := \left\{(f^\varepsilon_\dagger,H_{\dagger,f}^\varepsilon) \, \middle| \, f \in C_l^\infty(E), \varepsilon \in (0,1) \right\}.
			\end{equation*} 
			\item For $f \in C_u^\infty(E)$ and $\varepsilon \in (0,1)$  set 
			\begin{gather*}
				f^\varepsilon_\ddagger := (1+\varepsilon) f - \varepsilon \Upsilon, \\
				H_{\ddagger,f}^\varepsilon(x) := (1+\varepsilon) \bfH f(x) - \varepsilon C_\Upsilon,
			\end{gather*}
			and set
			\begin{equation*}
				H_\ddagger := \left\{(f^\varepsilon_\ddagger,H_{\ddagger,f}^\varepsilon) \, \middle| \, f \in C_u^\infty(E), \varepsilon \in (0,1) \right\}.
			\end{equation*} 
		\end{itemize}
	\end{definition}

	\begin{definition} \label{definition:H1H2}
		\begin{itemize}
			\item For $f \in C_l^\infty(E)$ , $\varepsilon \in (0,1)$ and $\phi$ such that $e^\phi \in \cD(A^\mathrm{fast})$ set 
			\begin{gather*}
				f^\varepsilon_1 := (1-\varepsilon) f + \varepsilon \Upsilon, \\
				H^\varepsilon_{1,f,\phi}(x,z) :=
				(1-\varepsilon) \left( V_{x,\nabla f(x)}(z) + e^{-\phi(z)} A^\mathrm{fast}_x e^{\phi}(z)\right) + \varepsilon C_\Upsilon ,
			\end{gather*}
			and set
			\begin{equation*}
				H_1 := \left\{(f^\varepsilon_1,H^\varepsilon_{1,f,\phi}) \, \middle| \, f \in C_l^\infty(E), \varepsilon \in (0,1), \phi: \, e^\phi \in \cD(A^\mathrm{fast})  \right\}.
			\end{equation*} 
			\item For $f \in C_u^\infty(E)$, $\varepsilon \in (0,1)$ and $\phi$ such that $e^\phi \in \cD(A^\mathrm{fast})$ set 
			\begin{gather*}
				f^\varepsilon_2 := (1+\varepsilon) f - \varepsilon \Upsilon, \\
				H^\varepsilon_{2,f,\phi}(x,z) :=
				(1+\varepsilon) \left( V_{x,\nabla f(x)}(z)  + e^{-\phi(z)} A^\mathrm{fast}_x e^{\phi}(z) \right) - \varepsilon C_\Upsilon,
			\end{gather*}
			and set
			\begin{equation*}
				H_2 := \left\{(f^\varepsilon_2,H^\varepsilon_{2,f,\phi}) \, \middle| \, f \in C_u^\infty(E), \varepsilon \in (0,1), \phi: \, e^\phi \in \cD(A^\mathrm{fast})  \right\}.
			\end{equation*}
		\end{itemize}
	\end{definition}

	\subsection{The comparison principle for \texorpdfstring{$H_\dagger$}{Hdagger} and \texorpdfstring{$H_\ddagger$}{Hddagger}}
	
	The next theorem contains the comparison principle for $H_\dagger$ and $H_\ddagger$. This result is the key statement obtained in \cite{KrSc19}. We specialize it to our setting.

	\begin{theorem} \label{theorem:comparison_from_otherpaper}
		Let $h_1,h_2 \in C_b(E)$ and $\lambda >0$. Let $u$ be any subsolution to $f - \lambda H_\dagger f = h_1$ and let $v$ be any supersolution to $f - \lambda H_\ddagger f = h_2$. Then we have that 
		\begin{equation*}
			\sup_x u(x) - v(x) \leq \sup_x h_1(x) - h_2(x).
		\end{equation*}
	\end{theorem}
	
	\begin{proof}
		This result follows from~\cite[Proposition 3.4]{KrSc19}. Recall that we use $\Theta = \cP(F)$ and $\Lambda(x,p,\pi) = \int V_{x,p}(z) \pi(\dd z)$. To apply the proposition, we have to verify \cite[Assumptions 2.14 and 2.15]{KrSc19}. These, however, are directly implied by Assumptions \ref{assumption:results:regularity_of_V} and \ref{assumption:results:regularity_I} of this paper.
	\end{proof}
	
	\subsection{Transfer of sub- and supersolutions based on the solution of an eigenvalue problem} \label{subsection:implications_from_ev}

	\begin{lemma}\label{lemma:viscosity_solutions_arrows_based_on_eigenvalue}
		Fix $\lambda > 0$ and $h \in C_b(E)$. 
		\begin{enumerate}[(a)]
			\item Every subsolution to $f - \lambda H_1 f = h$ is also a subsolution to $f - \lambda H_\dagger f = h$.
			\item Every supersolution to $f - \lambda H_1 f = h$ is also a supersolution to $f - \lambda H_\ddagger f = h$.
		\end{enumerate}
	\end{lemma}

	The definition of viscosity solutions,  Definition \ref{definition:viscosity_solutions}, is written down in terms of the existence of a sequence of points that maximizes $u-f$ or minimizes $v-f$. To prove the lemma above, we would like to have the subsolution and supersolution inequalities for any point that maximizes or minimizes the difference. This is achieved by the following auxiliary lemma.
	
	\begin{lemma} \label{lemma:strong_viscosity_solutions}
		Fix $\lambda > 0$ and $h \in C_b(E)$. 
		\begin{enumerate}[(a)]
			\item Let $u$ be a subsolution to $f - \lambda H_1 f = h$, then for all $(f,g) \in H_1$ and $x_0 \in E$ such that
			\begin{equation*}
				u(x_0) - f(x_0) = \sup_x u(x) - f(x)
			\end{equation*}
			there exists a $z \in F$ such that 
			\begin{equation*}
				u(x_0) - \lambda g(x_0,z)  \leq h(x_0).
			\end{equation*}
			\item Let $v$ be a supersolution to $f - \lambda H_2 f = h$, then for all $(f,g) \in H_2$ and $x_0 \in E$ such that
			\begin{equation*}
				v(x_0) - f(x_0) = \inf_x v(x) - f(x)
			\end{equation*}
			there exists a $z \in F$ such that 
			\begin{equation*}
				v(x_0) - \lambda g(x_0,z)  \geq h(x_0).
			\end{equation*}
		\end{enumerate}
	\end{lemma}
	
	The following proof is inspired by \cite[Lemma 9.9]{FK06}.
	
	\begin{proof}
		We only prove (a). Let $u$ be a subsolution to $f - \lambda H_1 f = h$ and $(f,g) \in H_1$. For later use, we explicitly give the form of $g$. Fix $\varepsilon \in (0,1)$ and $\phi$ such that
		\begin{equation} \label{eqn:g_in_optimizing_point_lemma}
			g = H_{1,f,\phi}^\varepsilon
		\end{equation}
		as in Definition \ref{definition:H1H2}.

		\smallskip
		
		\textit{Step 1}: We start with a preliminary observation based on the compactness of the level sets of $f$ in $E \times F$. The compactness implies that the optimizing sequence $(x_n,z_n)$ in the definition of the notion of a viscosity subsolution allows for a converging subsequence. Using the continuity of all functions involved, we thus find the existence of a point $(x_0,z) \in E \times F$ such that
		\begin{gather}
			u(x_0) - f(x_0) = \sup_x u(x) - f(x), \label{eqn:subsol_in_point1} \\
			u(x_0) - \lambda g(x_0,z)  \leq h(x_0). \label{eqn:subsol_in_point2}
		\end{gather}

		\textit{Step 2}: We proceed by showing that for any $x_0$ such that \eqref{eqn:subsol_in_point1} is satisfied there is some $z$ such that \ref{eqn:subsol_in_point2} holds. Thus, let $x_0$ be such that $u(x_0) - f(x_0) = \sup_x u(x) - f(x)$. Pick a function $\hat{f} \in C_{cc}^\infty(E)$ such that $\hat{f}(x_0)  = 0$ and $\hat{f}(x) > 0$ for $x \neq x_0$. Set $f_0 = f+ \hat{f}$ and let $g_0$ be such that 
		\begin{equation*}
			g_0 = H_{1,f_0,\phi}^\varepsilon
		\end{equation*}
		where $\varepsilon$ and $\phi$ are as in \eqref{eqn:g_in_optimizing_point_lemma}. It follows that $(f_0,g_0) \in H_1$. 
		
		\smallskip

		By construction $x_0$ is the unique point such that $u(x_0) - f_0(x_0) = \sup_x u(x) - f_0(x)$, so that by the sub-solution property studied in step 1 for $(f_0,g_0)$ instead of $(f,g)$, we find the existence of $z_0$ such that
		\begin{equation*}
			u(x_0) - \lambda g(x_0,z_0) = u(x_0) - \lambda g_0(x_0,z_0) \leq h_0(x_0).
		\end{equation*}
		As $\nabla f_0(x_0) = \nabla f(x_0)$ and $g(x_0,z_0)$ and $g_0(x_0,z_0)$ only depend on $f$ and $f_0$ via their derivatives at $x_0$ it follows that $g_0(x_0,z_0) = g(x_0,z_0)$. We thus find
		\begin{equation*}
			u(x_0) - \lambda g(x_0,z) = u(x_0) - \lambda g_0(x_0,z) \leq h_0(x_0)
		\end{equation*}
		establishing the claim.
	\end{proof}

	\begin{proof}[Proof of Lemma \ref{lemma:viscosity_solutions_arrows_based_on_eigenvalue}]
		We only prove the subsolution statement. Fix $\lambda > 0$ and $h \in C_b(E)$. 
		
		Let $u$ be a subsolution of $f - \lambda H_1 f = h$. We prove it is also a subsolution of $f - \lambda H_\dagger f = h$. Let $f^\varepsilon_1 = (1-\varepsilon)f + \varepsilon \Upsilon \in \cD(H_1)$ and let $x_0$ be such that
		\begin{equation*}
			u(x_0) - f^\varepsilon_1(x_0) = \sup_x u(x) - f_1^\varepsilon(x).
		\end{equation*}
		For each $\delta > 0$ we find by Assumption \ref{assumption:compact_setting:eigen_value} a function $\phi_\delta$ such that $e^{\phi_\delta} \in \cD(A^\mathrm{fast})$ and
		\begin{equation} \label{eqn:inequality_based_on_ev}
			\cH(x,p) \geq V_{x_0,\nabla f(x_0)}(z) - e^{-\phi_\delta(z)}\left(A_{x_0}^\mathrm{fast} e^{\phi_\delta}\right)(z) - \delta
		\end{equation}
		for all $z \in F$. As
		\begin{equation*}
			\left(f^\varepsilon_1, (1-\varepsilon) \left( V_{x,\nabla f(x)}(z) + e^{-\phi_\delta(z)} A^\mathrm{fast}_x e^{\phi_\delta(z)}\right) + \varepsilon C_\Upsilon \right) \in H_1,
		\end{equation*}
		we find by the subsolution property of $u$ and Lemma \ref{lemma:strong_viscosity_solutions} that there exists $z_0$ such that
		\begin{align*}
			h(x_0) & \geq u(x_0) - \lambda \left((1-\varepsilon) \left( V_{x_0,\nabla f(x_0)}(z_0) + e^{-\phi_\delta(z_0)} A^\mathrm{fast}_{x_0} e^{\phi_\delta(z_0)}\right) + \varepsilon C_\Upsilon\right) \\
			& \geq	u(x_0) - \lambda \left((1-\varepsilon)  \cH(x_0,\nabla f(x_0)) + \varepsilon C_\Upsilon\right) - \lambda(1-\varepsilon)\delta.
		\end{align*}
		where the second inequality follows by \eqref{eqn:inequality_based_on_ev}.
		
		Sending $\delta \rightarrow 0$ establishes that $u$ is a subsolution for $f - \lambda H_\dagger f = h$.
	\end{proof}

	\subsection{Transfer of sub- and supersolutions based on compact containment} \label{subsection:implications_from_compact_containment}

	The operator $H$ and $\bfH$ are related to $H_1, H_2$ and $H_\dagger,H_\ddagger$ by the following two Lemma's respectively.
	
	\begin{lemma} \label{lemma:viscosity_solutions_compactify1}
		Fix $\lambda > 0$ and $h \in C_b(E)$. 
		\begin{enumerate}[(a)]
			\item Every subsolution to $f - \lambda H f = h$ is also a subsolution to $f - \lambda H_1 f = h$.
			\item Every supersolution to $f - \lambda H f = h$ is also a supersolution to $f - \lambda H_2 f = h$.
		\end{enumerate}
	\end{lemma}
	
	\begin{lemma} \label{lemma:viscosity_solutions_compactify2}
		Fix $\lambda > 0$ and $h \in C_b(E)$. 
		\begin{enumerate}[(a)]
			\item Every subsolution to $f - \lambda \bfH f = h$ is also a subsolution to $f - \lambda H_\dagger f = h$.
			\item Every supersolution to $f - \lambda \bfH f = h$ is also a supersolution to $f - \lambda H_\ddagger f = h$.
		\end{enumerate}
	\end{lemma}
	
	Lemma \ref{lemma:viscosity_solutions_compactify2} has been proven in Lemma 3.3 of \cite{KrSc19}. Of Lemma \ref{lemma:viscosity_solutions_compactify1}, we will only prove (a), and its proof is similar to that that of Lemma 3.3 (a) of \cite{KrSc19}.
	
	\begin{proof}
		Fix $\lambda > 0$ and $h \in C_b(E)$. Let $u$ be a subsolution to $f - \lambda H f = h$. We prove it is also a subsolution to $f - \lambda H_1 f = h$. Fix $\varepsilon \in (0,1)$, $\phi$ such that $e^\phi \in \mathcal{D}(A^\mathrm{fast})$, and $f \in C_{l}^\infty(E)$, so that $(f^\varepsilon_1,H^\varepsilon_{1,f,\phi}) \in H_1$. We will prove that there are $(x_n,z_n)$ such that
		\begin{gather} 
			\lim_n u(x_n) - f^\varepsilon_1(x_n) = \sup_x u(x) - f^\varepsilon_1(x),\label{eqn:proof_lemma_conditions_for_subsolution_first} \\
			\limsup_n u(x_n) - \lambda H^\varepsilon_{1,f,\phi}(x_n,z_n) - h(x_n) \leq 0. \label{eqn:proof_lemma_conditions_for_subsolution_second}
		\end{gather}
		% 		\sout{As $u -(1-\varepsilon)f$ is bounded from above and $\varepsilon \Upsilon$ has compact sublevel-sets, the sequence $x_n$ along which the first limit is attained can be assumed to lie in the compact set $K := \left\{x \, | \, \Upsilon(x) \leq \inf_x \varepsilon^{-1} \left(u(x) - (1-\varepsilon)f(x) \right)\right\}$. Set $M = \inf_x \varepsilon^{-1} \left(u(x) - (1-\varepsilon)f(x) \right)$.}
		
		We have that $M := \varepsilon^{-1} \sup_y u(y) - (1-\varepsilon)f(y) < \infty$ as $u$ is bounded and $f \in C_l(E)$. It follows that the sequence $x_n$ along which the limit in \eqref{eqn:proof_lemma_conditions_for_subsolution_first} is attained is contained in the compact set $K := \left\{x \, | \, \Upsilon(x) \leq M \right\}$.

		Let $\gamma : \bR \rightarrow \bR$ be a smooth increasing function such that
		\begin{equation*}
			\gamma(r) = \begin{cases}
				r & \text{if } r \leq M, \\
				M + 1 & \text{if } r \geq M+2.
			\end{cases}
		\end{equation*}
		Denote by $f_\varepsilon$ the function on $E$ defined by 
		\begin{equation*}
			f_\varepsilon(x) := \gamma\left((1-\varepsilon)f(x) + \varepsilon \Upsilon(x) \right) = \gamma(f_1^\varepsilon(x)).
		\end{equation*}
		By construction $f_\varepsilon$ is smooth and constant outside of a compact set and thus lies in $\cD(H) = C_{cc}^\infty(E)$. As $e^\phi \in \mathcal{D}(A^\mathrm{fast})$ we have by Assumption \ref{item:assumption:PI:domain} that also $e^{(1-\varepsilon)\phi} \in \mathcal{D}(A^\mathrm{fast})$. We conclude that $(f_\varepsilon, H_{f_{\varepsilon},(1-\varepsilon)\phi}) \in H$. 
		
		As $u$ is a viscosity subsolution for $f - \lambda Hf = h$ there exist $x_n \in K \subseteq E$ (by our choice of $K$) and $z_n \in F$ with
		\begin{gather}
			\lim_n u(x_n) - f_\varepsilon(x_n) = \sup_x u(x) - f_\varepsilon(x), \label{eqn:visc_subsol_sup} \\
			\limsup_n u(x_n) - \lambda H_{f_{\varepsilon},(1-\varepsilon)\phi}(x_n,z_n) - h(x_n) \leq 0. \label{eqn:visc_subsol_upperbound}
		\end{gather}
		As $f_\varepsilon$ equals $f_1^\varepsilon$ on $K$, we have from \eqref{eqn:visc_subsol_sup} that also
		\begin{equation*}
			\lim_n u(x_n) - f_1^\varepsilon(x_n) = \sup_x u(x) - f_1^\varepsilon(x),
		\end{equation*}
		
		establishing \eqref{eqn:proof_lemma_conditions_for_subsolution_first}. Convexity of $p \mapsto V_{x,p}$ and $\psi \mapsto e^{-\psi(z)}\left(A_x^\mathrm{fast} e^\psi\right)(z)$ yields for arbitrary $(x,z)$ the elementary estimate
		\begin{align*}
			H_{f_{\varepsilon},(1-\varepsilon)\phi}(x,z) = & V_{x,\nabla f_\varepsilon}(z) + e^{-(1-\varepsilon)\phi(z)} \left(A^\mathrm{fast}_x e^{(1-\varepsilon)\phi}\right)(z) \\
			& \leq (1-\varepsilon) V_{x,\nabla f(x)}(z) + \varepsilon V_{x,\nabla \Upsilon(x)} + (1-\varepsilon)e^{-\phi(z)} \left(A^\mathrm{fast}_x e^{\phi}\right)(z) \\
			& = H^\varepsilon_{1,f,\phi}(x,z).
		\end{align*} 
		Combining this inequality with \eqref{eqn:visc_subsol_upperbound} yields
		\begin{multline*}
			\limsup_n u(x_n) - \lambda H^\varepsilon_{1,f,\phi}(x,z) - h(x_n) \\
			\leq \limsup_n u(x_n) - \lambda H_{f_{\varepsilon},(1-\varepsilon)\phi}(x_n,z_n) - h(x_n) \leq 0,
		\end{multline*}
		establishing \eqref{eqn:proof_lemma_conditions_for_subsolution_second}. This concludes the proof.
	\end{proof}
	%%%%%%%%%%%%	

	\section{Action-integral representations of the rate function} \label{section:variational_representation}
	In this section, we will establish the two representations for the rate function given in Theorem \ref{theorem:LDP_general_Lagrangian_two_represenatations}. 
	The results are based on three main steps.
	\begin{itemize}
		\item We again utilize the result summarized in Figure \ref{figure:Hamiltonians} below. So far we have only given the viscosity solutions $R(\lambda)h$ to $f - \lambda Hf$ that arises from the large deviation structure of Theorem \ref{theorem:abstract_LDP}. The figure shows that $R(\lambda)h$ is the unique function that is a sub- and supersolution to the equations $f - \lambda H_\dagger f = h$ and $f - \lambda H_\ddagger f = h$ respectively.
		\item We use variational methods of \cite[Chapter 8]{FK06} to construct a viscosity solution $\bfR(\lambda) h$ to $f - \lambda \bfH f = h$ based on a running cost in terms of a Lagrangian $\cL$ that is the Legendre transform of $\cH$.
		
		By  Figure \ref{figure:Hamiltonians}, we must have $R(\lambda) h = \bfR(\lambda)h$. Starting from the equality of resolvents we work to an equality for semigroups. Afterwards, we dualize to obtain the first variational representation for $J$.
		\item As a final ingredient, we apply results from convex analysis \cite{Ro70,HULe01,Pa97} to rewrite the Lagrangian in terms of a double optimization to obtain our second representation for $J$.
	\end{itemize}
	
	\begin{figure}[h]
		\centering
		\begin{tikzpicture}
			\matrix (m) [matrix of math nodes,row sep=1em,column sep=4em,minimum width=2em]
			{
				{ } & H_1 &[7mm] H_\dagger &[5mm] { } \\
				H & { } & { } & \bfH \\
				{ }  & H_2 & H_\ddagger & { } \\};
			\path[-stealth]
			(m-2-1) edge node [above] {sub} (m-1-2)
			(m-2-1) edge node [below] {super \qquad { }} (m-3-2)
			(m-1-2) edge node [above] {sub \qquad { }} (m-1-3)
			(m-3-2) edge node [below] {super \qquad { }} (m-3-3)
			(m-2-4) edge node [above] {\qquad sub} (m-1-3)
			(m-2-4) edge node [below] {\qquad super} (m-3-3);
			
			\begin{pgfonlayer}{background}
				\node at (m-2-3) [rectangle,draw=blue!50,fill=blue!20,rounded corners, minimum width=1cm, minimum height=2.5cm]  {comparison};
			\end{pgfonlayer}
		\end{tikzpicture}
		\caption{Relations between Hamiltonians}
		\label{figure:Hamiltonians}
	\end{figure}
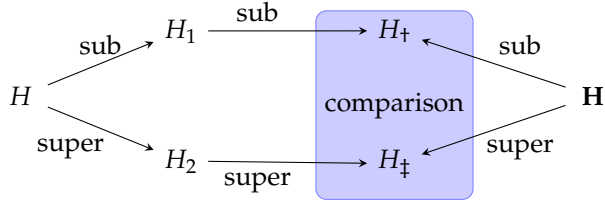
	Recall that $\bfH$ was based on the variational $\cH$ introduced in \eqref{eqn:def_variational_Hamiltonian}. The Legendre dual $\cL : E \times \bR^d \rightarrow [0,\infty]$ of $\cH$ is given by
	\begin{equation*}
		\cL(x,v) := \sup_{p\in\mathbb{R}^d} \left[\ip{p}{v} - \cH(x,p)\right].
	\end{equation*}
	In our new variational resolvent and semigroup, this Lagrangian plays the role of a running cost. In the following Theorem, $\cA\cC$ is the collection of absolutely continuous paths in $E$.
	%%%
	\begin{proposition} \label{proposition:variational_resolvent}
		Suppose that Assumptions \ref{assumption:results:regularity_of_V}, \ref{assumption:results:regularity_I} and \ref{assumption:Hamiltonian_vector_field} are satisfied. For each $\lambda > 0$ and $h \in C_b(E)$, let $\bfR(\lambda)h$ be given by
		\begin{equation*}
			\bfR(\lambda) h(x) = \sup_{\substack{\gamma \in \mathcal{A}\mathcal{C}\\ \gamma(0) = x}} \int_0^\infty \lambda^{-1} e^{-\lambda^{-1}t} \left[h(\gamma(t)) - \int_0^t \mathcal{L}(\gamma(s),\dot{\gamma}(s))\right] \, \dd t.
		\end{equation*}
		Then $\bfR(\lambda)h$ is the unique viscosity solution to $f - \lambda \bfH f = h$.
	\end{proposition}
	Before giving the proof, we repeat the definition of the subdifferential set of a convex functional given in \eqref{eqn:subdifferential}. Let $\Phi : \bR^d \rightarrow (-\infty,\infty]$ be convex. Then the subdifferential set of $\Phi$ is given by
	\begin{equation*}
		\partial_p \Phi(p_0) := \left\{
		\xi \in \mathbb{R}^d \, \middle| \, \Phi(p) \geq \Phi(p_0) + \xi \cdot (p-p_0) \quad (\forall p \in \mathbb{R}^d)
		\right\}.
	\end{equation*}
	\begin{proof}[Proof of Proposition \ref{proposition:variational_resolvent}]
		As above, this result follows as a consequence of results in \cite{KrSc19}. In this case we consider \cite[Theorem 2.8]{KrSc19}. To apply this result, we have to verify Assumptions \cite[Assumption 2.14]{KrSc19}, \cite[Assumption 2.15]{KrSc19} and \cite[Assumption 2.17]{KrSc19}. As above, the first two assumptions are implied by Assumptions \ref{assumption:results:regularity_of_V} and \ref{assumption:results:regularity_I} of this paper. 
		
		To establish Assumption \cite[Assumption 2.17]{KrSc19}, we need to verify that
		\begin{equation*}
			\partial_p \left(\int V_{x,p}(z) \pi(\dd z) \right) \subseteq T_E(x), \qquad \text{for any } \qquad \pi \in \cP(F).
		\end{equation*}
		This follows immediately from \cite[Theorem 3]{Pa97} applied with $\varepsilon = 0$ and Assumption \ref{assumption:Hamiltonian_vector_field}.
	\end{proof}

	By the results summarized in Figure \ref{figure:Hamiltonians}, we can therefore conclude that $R(\lambda)h = \bfR(\lambda)h$ for all $h \in C_b(E)$ and $\lambda > 0$. Next, consider the semigroup $V(t)$ of Theorem \ref{theorem:abstract_LDP} and the variational semigroup $\bfV$ on $C_b(E)$ defined by
	\begin{equation*}
		\bfV(t) f(x) := \sup_{\substack{\gamma \in \mathcal{A}\mathcal{C}\\ \gamma(0) = x}} f(\gamma(t)) -  \int_0^t \mathcal{L}(\gamma(s),\dot{\gamma}(s)).
	\end{equation*}
	
	\begin{proposition}
		Suppose that Assumptions~\ref{assumption:results:regularity_of_V} and~\ref{assumption:results:regularity_I} are satisfied for $\Lambda$ and~$\mathcal{I}$, and that $\mathcal{H}$ satisfies Assumption~\ref{assumption:Hamiltonian_vector_field}.
		Then $V(t)f = \bfV(t)f$ for all $t \geq 0$ and $f \in C_b(E)$.
	\end{proposition}
	
	\begin{proof}
		By \eqref{eqn:convergence_resolvent_to_semigroup_abstract_LDP}, there is some sequentially strictly dense set $D \subseteq C_b(E)$ such that for $f \in D$
		\begin{equation} \label{eqn:convergence_R_to_V}
			\lim_{m \rightarrow \infty} \vn{R\left(\frac{t}{m}\right)^m f - V(t)f} = 0.
		\end{equation}
		Similarly, we find by \cite[Lemma 8.18]{FK06} that for all $f \in C_b(E)$ and $x \in E$
		\begin{equation} \label{eqn:convergence_bfR_to_bfV}
			\lim_{m \rightarrow \infty} R\left(\frac{t}{m}\right)^m f(x) = \bfV(t)f(x).
		\end{equation}
		As Figure \ref{figure:Hamiltonians} implies that $R(\lambda)h$ = $\bfR(\lambda)h$ for $h \in C_b(E)$, we conclude from \eqref{eqn:convergence_R_to_V} and \eqref{eqn:convergence_bfR_to_bfV} that $V(t)f = \bfV(t)f$ for all $t$ and $f \in D$.
		
		\smallskip
		
		Now recall that $D$ is sequentially strictly dense by assumption so that equality for all $f \in C_b(E)$ follows if $V(t)$ and $\bfV(t)$ are sequentially continuous. The first statement follows by Theorems \cite[Theorem 7.10]{Kr19} and \cite[Theorem 6.1]{Kr20} on which Theorem \ref{theorem:LDP_general} is based. The second statement follows by Lemma \cite[Lemma 8.22]{FK06}. We conclude that $V(t)f = \bfV(t)f$ for all $f \in C_b(E)$ and $t \geq 0$.
	\end{proof}
	
	The argument of the proof of \cite[Theorem 8.14]{FK06} combined with the fact that $v \mapsto \cL(x,v)$ is convex leads to the following result.
	
	\begin{lemma} \label{lemma:representation_rate_function_initial}
		Suppose that Assumptions~\ref{assumption:results:regularity_of_V} and~\ref{assumption:results:regularity_I} are satisfied for $\Lambda$ and~$\mathcal{I}$, and that $\mathcal{H}$ satisfies Assumption~\ref{assumption:Hamiltonian_vector_field}. Then the rate function $J$ of Theorem \ref{theorem:LDP_general} can be rewritten as 
		\begin{equation*}
			J(\gamma) = \begin{cases}
				J_0(\gamma(0)) + \int_0^\infty \cL(\gamma(s),\dot{\gamma}(s)) \dd s & \text{if } \gamma \in \cA\cC, \\
				\infty & \text{otherwise}.
			\end{cases}
		\end{equation*}
	\end{lemma}
	
	The final step to derive Theorem \ref{theorem:LDP_general} from that of Theorem \ref{theorem:abstract_LDP} is to establish that we can rewrite $\cL$. This rewrite is a consequence of results in convex analysis and follows under much weaker assumptions: convexity of $\pi \mapsto \cI(x,\pi)$ and $p \mapsto V_{x,p}(z)$, which are satisfied in our setting.

	\begin{proposition} \label{proposition:representation_of_Lagrangian}
		Suppose that Assumptions \ref{assumption:results:regularity_of_V} and \ref{assumption:results:regularity_I} are satisfied.
		Let $\Phi(v,\pi)$ be the set of functions $w \in L^1(F,\pi)$ such that 
		\begin{equation*}
			\int w(z) \pi(\dd z) = v.
		\end{equation*}
		Then we have 
		\begin{equation*}
			\cL(x,v) = \inf_\pi \inf_{w \in \Phi(v,\pi)} \left\{\int \cL_z(x,w(z)) \pi(\dd z) + \cI(x,\pi)  \right\}.
		\end{equation*}
	\end{proposition}
	
	The proof below only uses results from convex analysis, see \cite{Ro70,HULe01,Pa97}. These results have been stated for $\bR^d$, which is also the setting to which we restrict ourselves in this paper. We believe, however, that the result should extend to a more general setting, but we were unable to find their generalizations in the literature on convex analysis.
	
	Before we prove Proposition \ref{proposition:representation_of_Lagrangian}, we start with two auxiliary lemmas.
	\begin{lemma} \label{lemma:DV_is_convex}
		Consider the context of Proposition \ref{proposition:representation_of_Lagrangian}.	For every $x \in E$, the map $\pi \mapsto \cI(x,\pi)$ is convex.
	\end{lemma}
	\begin{proof}
		Fix $x \in E$. As 
		\begin{equation*}
			\cI(x,\pi) = \sup_{\substack{\phi \in \cD(A^\mathrm{fast}) \\ \inf \phi > 0}} - \int \frac{A^\mathrm{fast}_x \phi(z)}{\phi(z)} \pi(\dd z)
		\end{equation*}
		is given as the supremum over linear maps, $\pi \mapsto \cI(x,\pi)$ is convex. 
	\end{proof}
	\begin{lemma} \label{lemma:constant_momentum}
		Consider the context of Proposition \ref{proposition:representation_of_Lagrangian}. Fix $\pi \in \cP(F)$ and denote 
		\begin{align*}
			\cH_\pi(x,p) & = \int V_{x,p}(z) \pi(\dd z), \\
			\cL_\pi(x,v) & = \sup_p  \ip{p}{v} - \int V_{x,p}(z) \pi(\dd z).
		\end{align*}
		
		Let $\Phi(v,\pi)$ be the set functions $w \in L^1(F,\pi)$ such that 
		\begin{equation*}
			\int w(z) \pi(\dd z) = v.
		\end{equation*}

		Suppose that $v \in \text{rel. int. dom } \cL_\pi(x,\cdot)$. Then:
		\begin{enumerate}[(a)]
			\item \label{item:convex_analysis_existence_p*} There is a $p^* \in \bR^d$ such that $v \in \partial_p \cH_\pi(x,p^*)$,
			\item \label{item:convex_analysis_optimal_speeds} There is a $w^* \in \Phi(v,\pi)$ such that $w^*(z) \in \partial_p V_{x,p^*}(z)$ $\pi$ almost surely.
			\item \label{item:equality_of_Lagrangians_in_relint} We have 
			\begin{multline*}
				\sup_{p} \ip{p}{v} - \int V_{x,p}(z) \pi(\dd z) \\
				= \inf_{\substack{w(z) \\ \int w \dd \pi = v}} \sup_{p(z)} \int   \ip{p(z)}{w(z)} - V_{x,p(z)}(z) \pi(\dd z),
			\end{multline*}
			where $\sup_{p(z)}$ is to be interpreted as the supremum over measurable functions $p : F \rightarrow \bR^d$.
		\end{enumerate}

	\end{lemma}
	
	For the proof of this lemma, we will use the notion of the relative interior of a convex set. If $A \subseteq \bR^d$ is a convex set, then $\text{rel. int. } A$ is the interior of $A$ inside the smallest affine hyperplane in $\bR^d$ that contains $A$. For a convex functional $\Phi : E \rightarrow (-\infty,\infty]$ the domain of $\Phi$, denoted by $\text{dom } \Phi$, is the set of points $x \in E$ where $\Phi(x) < \infty$.
	
	\begin{proof}[Proof of Lemma \ref{lemma:constant_momentum}]
		The Legendre transform of $\cH_\pi$ is equal to $\cL_\pi$. Since~$v \in \text{rel. int. dom } \cL_\pi(x,\cdot)$, we have by  \cite[Theorem 23.4]{Ro70} or \cite[Theorem E.1.4.2]{HULe01} that $\partial_v \cL_\pi(x,v)$ is non-empty. Let $p^* \in \partial_v \cL_\pi(x,v)$. Then by  \cite[Theorem 23.5]{Ro70} or \cite[Propposition E.1.4.3]{HULe01}, we have $v \in \partial_p H_\pi(x,p^*)$ establishing \ref{item:convex_analysis_existence_p*}.
		
		By Theorem 3 of \cite{Pa97} applied for $\varepsilon = 0$, we find a $\pi$ integrable function $w^*$ such that $\int w^*(z) \pi(\dd z) = v$ and $w(z) \in \partial_p V_{x,p^*}(z)$ $\pi$ almost surely, establishing \ref{item:convex_analysis_optimal_speeds}.
		
		\smallskip
		
		We proceed with the proof of \ref{item:equality_of_Lagrangians_in_relint}. First of all, note that
		\begin{multline*}
			\sup_{p} \ip{p}{v} - \int V_{x,p}(z) \pi(\dd z) \\
			\leq \inf_{\substack{w(z) \\ \int w \dd \pi = v}} \sup_{p(z)} \int \ip{p(z)}{w(z)} - V_{x,p(z)}(z) \pi(\dd z).
		\end{multline*}
		For the other inequality, note that 
		\begin{align*}
			&  \inf_{\substack{w(z) \\ \int w \dd \pi = v}} \sup_{p(z)} \int \ip{p(z)}{w(z)} - V_{x,p(z)}(z) \pi(\dd z) \\
			& \quad \leq \sup_{p(z)} \int \ip{p(z)}{w^*(z)} - V_{x,p(z)}(z) \pi(\dd z) \\
			& \quad = \int \ip{p^*}{w^*(z)} - V_{x,p^*}(z) \pi(\dd z) \\
			& \quad = \ip{p^*}{v} - \int V_{x,p^*}(z) \pi(\dd z) \\
			& \quad \leq \sup_p \ip{p}{v} - \int V_{x,p}(z) \pi(\dd z)
		\end{align*}
		where we used in line 3 that $w^*(z) \in \partial_p V_{x,p^*}(z)$ $\pi$ almost surely and  \cite[Theorem 23.5]{Ro70} or \cite[Propposition E.1.4.3]{HULe01}.
	\end{proof}
	
	\begin{proof}[Proof of Proposition \ref{proposition:representation_of_Lagrangian}]
		First of all, note that
		\begin{align}
			\cL(x,v) & = \sup_p \ip{p}{v} - \cH(x,p)  \notag \\
			& = \sup_p \inf_{\pi \in \cP(F)} \ip{p}{v} - \int V_{x,p}(z) \pi(\dd z) + \cI(x,\pi) \notag \\
			& = \inf_{\pi \in \cP(F)} \sup_p  \ip{p}{v} - \int V_{x,p}(z) \pi(\dd z) + \cI(x,\pi) \label{eqn:sion_lagrangian}
		\end{align}
		by Sion's minimax lemma, using that by Lemma \ref{lemma:DV_is_convex} the map $\pi \mapsto \cI(x,\pi)$ is convex.
		
		Fix $\pi \in \cP(F)$. Denote by
		\begin{align*}
			\cL_\pi(x,v) & = \sup_p  \ip{p}{v} - \int V_{x,p}(z) \pi(\dd z), \\
			\widehat{\cL}_\pi(x,v) & = \inf_{w: \int w(z) \pi(\dd z) = v} \int \cL_z(x,w(z)) \pi(\dd z).
		\end{align*}
		
		Thus, by \eqref{eqn:sion_lagrangian}, our proposition follows if for all $(x,v)$ and $\pi$ we have
		\begin{equation} \label{eqn:equality_Lagrangians_pi}
			\cL_\pi(x,v) = \widehat{\cL}_\pi(x,v).
		\end{equation}
		Fix $(x,v)$ and $\pi$.
		
		\smallskip
		
		\textit{Step 1:} We establish $\cL_\pi(x,v) \leq \widehat{\cL}_\pi(x,v)$. For any integrable function $z \mapsto w(z)$ such that $\int w(z) \pi(\dd z) = v$, we have
		\begin{equation*}
			\cL_\pi(x,v) = \sup_p \int \ip{p}{w(z)} - V_{x,p}(z) \pi(\dd z) 
		\end{equation*}
		implying that
		\begin{equation}
			\cL_\pi(x,v) = \inf_{w: \int w(z) \pi(\dd z) = v} \sup_p \int \ip{p}{w(z)} - V_{x,p}(z) \pi(\dd z) \leq \widehat{\cL}_\pi(x,v) 
		\end{equation}
		by taking the supremum over $p$ inside the integral. We conclude that $\cL_\pi(x,v) \leq \widehat{\cL}_\pi(x,v)$.
		
		\smallskip
		
		\textit{Step 2:} We now establish that if $v \in \text{rel. int. dom } \cL(x,\cdot)$ then $\cL_\pi(x,v) = \widehat{\cL}_\pi(x,v)$. Indeed, by Lemma \ref{lemma:constant_momentum} \ref{item:equality_of_Lagrangians_in_relint}, we have
		\begin{align*}
			\cL_\pi(x,v) & = \sup_{p}   \ip{p}{v} - \int V_{x,p(z)}(z) \pi(\dd z) \\
			& = \inf_{\substack{w(z) \\ \int w \dd \pi = v}} \sup_{p(z)} \int   \ip{p(z)}{w(z)} - V_{x,p(z)}(z) \pi(\dd z) \\
			& = \inf_{\substack{w(z) \\ \int w \dd \pi = v}}  \int   \cL_z(x,w(z)) \pi(\dd z) \\
			& = \widehat{\cL}_\pi(x,v).
		\end{align*}

		\textit{Step 3:} We now establish that $\cL_\pi(x,v) = \widehat{\cL}_\pi(x,v)$. By step 1, we have $\text{rel. int. dom } \widehat{\cL}_\pi(x,\cdot) \subseteq \text{rel. int. dom } \cL_\pi(x,\cdot)$ as the Lagrangians are ordered point-wise. By step 2, we have that if $v \in \text{rel. int. dom } \cL_\pi(x,\cdot)$, then $v \in \text{dom } \cL_\pi(x,\cdot)$. This implies that $\text{rel. int. dom } \widehat{\cL}_\pi(x,\cdot) = \text{rel. int. dom } \cL_\pi(x,\cdot)$ and that $\cL_\pi(x,\cdot) = \widehat{\cL}_\pi(x,\cdot)$ on this set. We conclude that $\cL_\pi(x,v) = \widehat{\cL}_\pi(x,v)$ by \cite[Corollary 7.3.4]{Ro70} (this can also be derived from \cite[Proposition B.1.2.6]{HULe01}). 
	\end{proof}
	%%%%%%%%%%%%

	\section{Mean-field interacting particles coupled to fast diffusion}
	\label{SF:sec:mean-field-fast-diffusion}
	%%%%%%%%%%%%%

	In this section, we provide a large-deviation result for mean-field interacting jump processes coupled to a fast diffusion process. Concretely, we take the simultaneous limit of infinitely many particles and infinite time-scale separation, and are interested in the large deviations of the empirical density-flux pairs of the mean-field system.
	
	\subsection{The setting of weakly interacting jump processes coupled to a fast process}

	For formulating the large-deviation result (Theorem~\ref{SF:thm:LDP-mean-field}), we first introduce the processes~$X_n(t)$ (Eq.~\eqref{SF:eq:def:density-flux} below) and~$Z_n(t)$ (Eq.~\eqref{SF:eq:intro-mean-field:fast-generator} below) independently from one another, and then consider the coupling. 
	We start with describing the mean-field system.

	\paragraph{The slow process: weakly interacting jump processes.}

	We consider a system of~$n$ jump processes 
	\begin{equation*}
		Y_n(t) := (Y_{n,1}(t),\dots,Y_{n,n}(t)) \in \{1,\dots,q\}^n.
	\end{equation*}
	on a finite state space. We assume that the processes are fully exchangeable, jump one-by-one, and interact weakly: their jump rates depend on their empirical measure
	\begin{equation*}
		\mu_n(t) = \mu_n(Y_n(t)) := n^{-1} \sum_{i=1}^n \delta_{Y_{n,i}(t)}
	\end{equation*}
	
	Our aim is to study the large deviations of the trajectory of empirical measures $\mu_n(t)$ as $n$ gets large. Following 
	\cite{BeChFaGa18,Re18,PaRe19,Kr17}, we will include the one-way fluxes (level 2.5 large deviations) as their inclusion gives greater insight into the problem at hand, and simplifies greatly the Lagrangian.
	
	Denote by 
	\begin{equation*}
		\Gamma := \left\{(a,b) \in \{1,\dots,q\}^2 \, \middle| \, a \neq b\right\}  
	\end{equation*}
	the set of one-way edges in $\{1,2,\dots,q\}$. Denote by $t \mapsto W_{n,i}(t) \in \bN^\Gamma$ the process that counts the number of times the i'th particle jumps over each bond,
	\begin{equation*}
		W_{n,i}(t)(a,b) := \#\left\{0\leq s \leq t \, \middle| \, \left(Y_{n,i}(s-), Y_{n,i}(s)\right) = (a,b) \right\}.
	\end{equation*}
	We regard~$W_{n,i}(t)$ as a vector taking values in~$\mathbb{N}^\Gamma$. The average fluxes over all bonds is captured by the \emph{empirical flux}~$W_{n}$ defined as
	\begin{equation*}
		W_n(t) :=\frac{1}{n}\sum_{i=1}^n W_{n,i}(t).
	\end{equation*}
	To make the connection with the general results of Section \ref{sec:results:general-framework}, we will choose as a slow process the pair empirical density and fluxes: 
	\begin{equation}\label{SF:eq:def:density-flux}
		X_n(t) := \left(\mu_n(Y_n(t)),W_n(t)\right) \in E:= \cP(\{1,\dots,q\}) \times [0,\infty)^\Gamma.
	\end{equation}

	\paragraph{The fast process: drift-diffusion}
	%%%%%%%%%%%%
	The process $Z_n(t)$ is a drift-diffusion process on the flat torus~$\mathbb{T}^m$. 
	While all arguments that follow still hold true on a closed, smooth, compact manifold~$F$, we do not consider this generalization here to avoid deviating from our main goal.

	%%%%%%%%%%%
	\paragraph{The generator of the coupled slow-fast system}
	%%%%%%%%%%%

	Above, we described two processes: the density-flux process~$X_n(t)$, and the drift diffusion $Z_n(t)$. We now state their generator, in which we make explicit the coupling between $X_n$ and $Z_n$.
	
	\smallskip
	
	We introduce first the state space. First, we write $\cP_n:=\{(1/n)\sum_{i=1}^n\delta_{q_i}\,:\,q_i\in\{1,\dots,q\}\}$ the subset of $\cP(\{1,\dots,q\}$ of~$n$-atomic measures. We find that $E_n = P_n \times (\frac{1}{n} \bN)^\Gamma \times \bT^d$. 
	
	\smallskip
	
	We start with the generator of the slow process $X_n(t)$. If~$\mu_n(Y_n(t))=:\mu$ and $Z_n(t) = z$, then the transition of a particle from~$a$ to~$b$ occurs at rate~$r_n(a,b,\mu,z)$. The number of particles in state~$a$ is $n\cdot \mu_a$. Hence the rate at which the configuration~$\mu$ transitions to the configuration $\mu + (\delta_b-\delta_a)/n$ is given by~$n\cdot \mu_a\cdot r_n(a,b,\mu,z)$. Simultaneous with this transition, the empirical flux changes from $w$ to $w +\frac{1}{n}\delta_{(a,b)}$. Therefore, the generator~$A_{n,z}^\mathrm{slow}:C_b(E_n)\to C_b(E_n)$ of the jump process~$X_t^n$ is 
	\begin{equation}\label{SF:eq:intro-mean-field:slow-generator}
		A_{n,z}^\mathrm{slow} f(x) = \sum_{a,b;a\neq b} n\cdot \mu_a\cdot r_n(a,b,\mu,z)\left[f(x_{a\to b}^n)-f(x)\right],
	\end{equation}
	where for a state~$x=(\mu,w)\in E_n$, we denote by~$x_{a\to b}^n$ the state after the jump. Since after the jump, exactly one particle has changed its state from~$a$ to~$b$,
	\begin{equation*}
		x_{a\to b}^n = \left(\mu + \frac{1}{n}(\delta_b-\delta_a), w+\frac{1}{n}\delta_{(a,b)}\right).
	\end{equation*}
	
	\smallskip
	
	The generator $A_{n,x}^\mathrm{fast}$ of the fast process $Z_n(t)$ is a second-order uniformly-elliptic differential operator with a domain $\mathcal{D}(A^\mathrm{fast})$ which contains $C^2(F)$  and such that for each choice of local coordinates and $f \in C^2(F)$ we have
	\begin{equation}\label{SF:eq:intro-mean-field:fast-generator}
		A_{n,x}^\mathrm{fast} f(z) = \sum_{i=1}^m b_n^i(z)\partial_i f(z) + \sum_{ij=1}^ma_n^{ij}(z)\partial_i\partial_j f(z).
	\end{equation} 
	$a_n(z) = \sigma_n(z)\sigma_n(z)^T$ are symmetric positive-definite matrices and the~$b_n(z)$ are vector fields. For details on the construction of the process from the operator, we refer to~\cite[Theorem~IV.6.1]{IkWa14} and the discussion thereafter.
	
	\smallskip
	
	To obtain a slow-fast system, we let the diffusion process run on the time-scale of order~$n$. As a consequence, the generator~$A_n$ of the couple~$(X_n(t),Z_n(t))$ is
	\begin{equation}\label{SF:eq:intro-mean-field:generator-slow-fast}
		A_n f(x,z) := A_{n,z}^\mathrm{slow}f(\cdot,z)(x) + n\cdot  A_{n,x}^\mathrm{fast}f(x,\cdot)(z).
	\end{equation}

	The following regularity condition is imposed in order to ensure that we obtain a Feller-continuous process~$(X_n(t),Z_n(t))$ solving the martingale problem~\cite[Theorem~2.1, Section~2.5 and Theorem~2.18]{YiZh09}.
	\begin{condition}[Regularity]\label{SF:condition:mean-field:reg-coefficients}
		For each~$i,j\in 1$,~$n=1,2,\dots$, we have:
		\begin{enumerate}[label=(\arabic*)] 
			\item For each~$x\in E$,~$a_n^{ij}(x,\cdot)\in C^2(F)$ and~$b_n^i(x,\cdot)\in C^1(F)$. 
			\item There is a constant~$C>0$ such that $\ip{a_n(x,z)\xi}{\xi}\geq C|\xi|^2$ for all~$\xi\in T_z F$ and for all~$(x,z)\in E_n\times F$.
			\item For each~$(a,b)\in\Gamma$, the jump rates~$r_n(a,b,\mu,z)$ depend continuously on~$(\mu,z)$, and $r_n(a,b,\mu,\cdot)\in C^1(F)$ for each~$\mu\in P_n$. 
		\end{enumerate}
	\end{condition}

	%%%%%%%%%%%
	\subsection{Large deviations for weakly interacting jump processes coupled to a fast diffusion}
	%%%%%%%%%%%

	We aim take the limit~$n\to\infty$ and study the large deviations of the trajectory $t \mapsto X_n(t)$ of our density-flux process on the path-space $D_E(\bR^+)$.
	
	\smallskip
	
	To be able to do this, we assume in the following two assumptions that the $n$-dependent rates and coefficients converge in an appropriate way as $n$ goes to infinity and that their limits have appropriate (weak) regularity conditions. We will additionally assume that the limit $r$ of the jump rates $r_n$ on each bond is either $0$ or bounded away from $0$.
	
	Note in particular that in next assumption, we assume that the limit $r$ of the jump rates $r_n$ is only continuous and not Lipschitz continuous as is often used in limits of mean-field interacting jump processes, see \cite{DuRaWu16,PaRe19,Kr17,BuDuFiRa15a,BuDuFiRa15b}.

	\begin{assumption}[Convergence of rates]\label{SF:mean-field:conv-rates}
		There is a kernel~$r=r(a,b,\mu,z)$ such that for each edge~$(a,b)\in\Gamma$,
		\begin{equation*}
			\lim_{n\to\infty}\sup_{\mu\in \mathcal{P}_n}\sup_{z\in F}\left| r_n(a,b,\mu,z)-r(a,b,\mu,z)\right| = 0.
		\end{equation*}
		There are constants~$0<r_\mathrm{min}\leq r_{\mathrm{max}}<\infty$ such that for all edges~$(a,b)\in\Gamma$ satisfying~$\sup_{\mu,z}r(a,b,\mu,z)>0$, we have
		\begin{equation*}
			r_\mathrm{min}\leq \inf_{\mu,z}r(a,b,\mu,z) \leq \sup_{\mu,z}r(a,b,\mu,z) \leq r_\mathrm{max}.
		\end{equation*}
	\end{assumption}

	\begin{assumption}[Convergence of coefficients]\label{SF:mean-field:conv-coeff}
		For each~$i,j$, there are functions~$b^i$ and~$\sigma^{ij}$ on~$E\times F$ such that whenever~$x_n=(\mu_n,w_n)\to (\mu,w)$, then
		\begin{equation*}
			\|b^i(\mu,\cdot)-b_n^i(x_n,\cdot)\|_{F}\to 0\quad\text{and}\quad \|a^{ij}(\mu,\cdot)-a_n^{ij}(x_n,\cdot)\|_{F}\to 0,
		\end{equation*}
		where~$\|g\|_F=\sup_F|g|$. The maps~$\mu\mapsto a^{ij}(\mu,\cdot)$ and $\mu \mapsto b^i(\mu,\cdot)$ are continuous as functions from~$\mathcal{P}(\{1,\dots,q\})$ to~$C(F)$ equipped with the uniform norm.
	\end{assumption}
	
	Before stating the large deviation result, we introduce the limiting objects that appear in the variational expression for the rate function.
	\begin{itemize}
		\item The slow Hamiltonian; for~$(x,p)\in E\times \mathbb{R}^d$,
		\begin{equation*}
			V_{x,p}(z) = \sum_{ab} \mu_a  r(a,b,\mu,z)\left[e^{p_b-p_a + p_{ab}}-1\right].
		\end{equation*}
		\item The Donsker-Varadhan functional; for~$x\in E$,
		\begin{equation*}
			\mathcal{I}(\mu,\pi) = -\inf_{\substack{u\in C^2(F) \\ \inf u>0}} \int_F \frac{A_\mu^\mathrm{fast}u}{u}\,\dd\pi,
		\end{equation*}
		where~$A_\mu^\mathrm{fast}u(z) := \sum_i b^n(\mu,z)\partial_i u(z) + \sum_{ij} a^{ij}(\mu,z)\partial_i u(z)\partial_j u(z)$. 
	\end{itemize}
	Finally, denote by $S(\alpha | \beta)$ the map
	\begin{equation*}
		S(\alpha \, | \, \beta) := \begin{cases}
			\beta & \text{if } \alpha = 0, \\
			\alpha \log \left(\alpha/\beta\right) - (\alpha-\beta) & \text{if } \alpha \neq 0, \beta \neq 0, \\
			+\infty & \text{if } \alpha \neq 0, \beta = 0.
		\end{cases}
	\end{equation*}
	
	Theorems \ref{theorem:LDP_general} and \ref{theorem:LDP_general_Lagrangian_two_represenatations} applied to present context, yields the following large deviation result.

	\begin{theorem}[Large deviations of the density-flux process] \label{SF:thm:LDP-mean-field}
		Let~$(X_n,Z_n)$ be the Markov process with generator~\eqref{SF:eq:intro-mean-field:generator-slow-fast}.
		Suppose that Assumptions~\ref{SF:mean-field:conv-rates} and~\ref{SF:mean-field:conv-coeff} hold true and that $X_n(0)$ satisfies a large-deviation principle with good rate function $J_0:E\to[0,\infty]$ on $E = \cP(\{1,\dots,q\}) \times [0,\infty)^\Gamma$.
		\smallskip
		
		Then $\{X_n\}_{n = 1,2\dots}$ satisfies a large-deviation principle on $D_{E}(\bR^+)$ with good rate function $J$:
		\begin{equation*}
			J(\gamma) = \begin{cases}
				J_0(\gamma(0)) + \int_0^\infty \cL(\gamma(s),\dot{\gamma}(s)) \dd s & \text{if } \gamma \in \cA\cC, \\
				\infty & \text{otherwise}.
			\end{cases}
		\end{equation*}
		$\cL$ has two representations:
		\begin{description}
			\item[Dual of the principal eigenvalue] The Lagrangian $\cL$ is given by
			\begin{equation*}
				\cL(x,v) =  \sup_p \left\{ \ip{p}{v} - \cH(x,p) \right\}, 
			\end{equation*}
			where
			\begin{equation*}
				\sup_\pi \left\{\int V_{x,p}(z) \pi(\dd z) - \cI(\mu,\pi) \right\}.
			\end{equation*}
			\item[Optimizing over velocities] For a path~$\gamma:[0,\infty)\to E$,~$\gamma=(\mu,w)$, the Lagrangian~$\mathcal{L}$ is finite only if $\partial_t\mu_a = \sum_b \partial_t (w_{ba}-w_{ab})$. If this is the case
			\begin{multline*}
				\cL\left(\gamma,\partial_t\gamma\right) \\
				= \inf_{\pi \in \cP(F)} \inf_{u \in \Phi(\partial_t w,\pi)}  \bigg\{ \sum_{(a,b) \in \Gamma} \int_F S(u_{ab}(z)\,|\, \mu_a  r(a,b,\mu,z)) \pi(\dd z) 
				\\+ \cI(\mu,\pi) \bigg\},
			\end{multline*}
			where $\Phi(\partial_t w,\pi)$ is the set of measurable functions $u_{ab}(z)$ for $z \in F$ and $(a,b) \in \Gamma$ such that~$\int u_{ab}(z) \pi(\dd z) = \partial_t w_{ab}$.
		\end{description}
		
	\end{theorem}
	
	We give the proof of this result in Section \ref{SF:sec:proof-mean-field} below. Before doing so, we derive from the large deviation principle the averaging principle.

	%%%%%%%%%%%%%
	\subsection{Averaging Principle}
	\label{SF:sec:averaging-principles}
	
	Related to the question of the large deviation principle, there is the question of a limiting result. It is well known that if the large deviation principle has a unique minimizer, then the dynamics has a limit. In the present context, the result one obtains is often called \textit{the averaging principle}.
	
	In the context that the rates $r(a,b,\mu,z)$ are Lipschitz as a function of $\mu$ was proven by Budhiraja, Dupuis, Fischer and Ramanan~\cite[Theorem~2.2]{BuDuFiRa15a} based on the classical result by Kurtz~\cite{Ku70b}. The limiting dynamics of the process of empirical measures $t \mapsto \mu_n(t)$ as $n \rightarrow \infty$ then satisfy the equation
	\begin{equation} \label{eqn:LLN_solution}
		\partial_t\mu_a = \int \sum_{b \neq a} \left[ \mu_b r(b,a,\mu,z) - \mu_a r(a,b,\mu,z) \right] \pi_\mu(\dd z) ,\quad \mu(0)=\mu_0,
	\end{equation}
	where $\pi_\mu$ is the stationary measure of the Markovian dynamics corresponding to $A_\mu^\mathrm{fast}$.
	
	\smallskip
	
	The result of Theorem \ref{SF:thm:LDP-mean-field} was obtained without a Lipschitz assumption on the rates. We can therefore, given the uniqueness of minimizers of the rate function, obtain the same limiting result.

	\begin{proposition}[Law of Large Number limit of mean-field interacting particles]
		Let~$X_n=(\mu_n,w_n)$ be the density-flux process from~$\eqref{SF:eq:def:density-flux}$ in the context of Theorem \ref{SF:thm:LDP-mean-field}. 
		Any minimizer of the rate function $J$ is a solution to \eqref{eqn:LLN_solution}. 
		If the rate function has a unique minimizer, the trajectory $t \mapsto \mu_n(t)$ of empirical measures converges uniformly on compact time intervals to this solution of \eqref{eqn:LLN_solution}
	\end{proposition}
	
	Clearly, \ref{eqn:LLN_solution} has a unique solution of $r$ is Lipschitz as a function of $\mu$. This would indeed imply that the rate function $J$ has a unique minimizer if $J_0$ has a unique minimizer, recovering the context of \cite{BuDuFiRa15a}.

	\begin{proof}
		As a consequence of the path-wise large-deviation principle of Theorem~\ref{SF:thm:LDP-mean-field}, any converging subsequences of the processes~$X^n$ converges a.s. to a minimizer~$x$ of the rate function~\cite[Theorem~A.2]{PeSc19}. We show that the density~$\mu$ of a minimizer~$x=(\mu,w)$ of the rate function~$J$ solves~\eqref{eqn:LLN_solution}. If~$J(\mu,w)=0$, then~$\mathcal{L}(x(t),\dot{x}(t))=0$ for a.e.~$t>0$, where the Lagrangian~$\mathcal{L}$ is given by
		\begin{multline*}
			\cL\left(x(t),\dot{x}(t)\right) = \\
			\inf_{\pi \in \cP(F)} \inf_{u \in \Phi(\dot{w},\pi)}  \left\{ \sum_{(a,b) \in \Gamma} \int S(u_{ab}(z)\,|\, \mu_a r(a,b,\mu(t),z)) \pi(\dd z) + \cI(\mu(t),\pi) \right\},
		\end{multline*}
		and by finiteness of the Lagrangian,
		\begin{equation}\label{SF:eq:proof:averaging-principle:mean-field:rho}
			\dot{\mu}_a(t) = \sum_{b \neq a} (\dot{w}_{ba}-\dot{w}_{ab}).
		\end{equation}
		As $\cI$ has compact level-sets and is strictly convex, the there exists a map $t \mapsto \pi_{\mu(t)}$ such that for all $t$ the measure $\pi_{\mu(t)}$ is the unique one such that $\cI(\mu(t),\pi_{\mu(t)}) = 0$.
		
		As all the terms in the Lagrangian are non-negative, it follows that for almost all $t$:
		\begin{multline*}
			0 = \cL\left(x(t),\dot{x}(t)\right) = \\
			\inf_{u \in \Phi(\dot{w},\pi)}  \left\{ \sum_{(a,b) \in \Gamma} \int_F S(u_{ab}(z)\,|\, \mu_a(t) r(a,b,\mu(t),z)) \pi_{\mu(t)}(\dd z) \right\}.
		\end{multline*}
		Since $S(\alpha \, | \, \beta) = 0$ if and only if $\alpha = \beta$ any optimizer~$u_{ab}(\cdot)$ satisfies
		\begin{equation}\label{SF:eq:proof:averaging-principle:mean-field:uab}
			u_{ab}(z)=\mu_a(t) r(a,b,\mu(t),z)\qquad \pi_\rho\,\mathrm{a}.\mathrm{e},
		\end{equation}
		and by definition of the set~$\Phi(\dot{w},\pi_\mu)$,
		\begin{equation}\label{SF:eq:proof:averaging-principle:mean-field:wab}
			\dot{w}_{ab}(t)=\int_F u_{ab}(z)\,\pi_{\mu(t)}(\dd z).
		\end{equation}
		Combining these equality's, we find for almost every $t$ that
		\begin{align*}
			\dot{\mu}_a(t) 
			&\overset{\eqref{SF:eq:proof:averaging-principle:mean-field:rho}}{=} \sum_{b\neq a} (\dot{w}_{ba}-\dot{w}_{ab})\\
			&\overset{\eqref{SF:eq:proof:averaging-principle:mean-field:wab}}{=} \sum_{b\neq a} \int (u_{ba}-u_{ab})\,\pi_{\mu(t)}(\dd z)\\
			&\overset{\eqref{SF:eq:proof:averaging-principle:mean-field:uab}}{=} \sum_{b\neq a} \left[\int \mu_b r(b,a,\mu(t),z)\pi_{\mu(t)}(\dd z) - \int \mu_a r(a,b,\mu(t),z)\pi_{\mu(t)}(\dd z)\right],
		\end{align*}
		which finishes the proof.
	\end{proof}

	%%%%%%%%%%%

	\subsection{Proof of Theorem \ref{SF:thm:LDP-mean-field}}
	\label{SF:sec:proof-mean-field}
	%%%%%%%%%%%%%
	In this section, we prove Theorem~\ref{SF:thm:LDP-mean-field} by verifying the assumptions of our general results, the large-deviation theorem and the action-integral representation. Hence we verify Assumptions~\ref{assumption:compact_setting:convergence-of-nonlinear-generators},~\ref{assumption:compact_setting:convergence-of-nonlinear-generators-external},~\ref{assumption:compact_setting:eigen_value},~\ref{assumption:results:regularity_of_V},~\ref{assumption:results:regularity_I} and~\ref{assumption:Hamiltonian_vector_field}.
	\smallskip
	
	To that end, recall the setting: the slow-fast process~$(X_n,Z_n)$ takes values in~$E_n\times F$, where we embed~$E_n$ into into~$E=\mathcal{P}(\{1,\dots,q\})\times[0,\infty)^\Gamma$ by using the identity map $\eta_n$. The set~$F$ is a finite-dimensional torus~$\mathbb{T}^,$. 
	
	The generator of the slow-fast system was given in \eqref{SF:eq:intro-mean-field:generator-slow-fast} as
	\begin{equation*}
		A_n f(x,z) := A_{n,z}^\mathrm{slow}f(\cdot,z)(x) + n\cdot  A_{n,x}^\mathrm{fast}f(x,\cdot)(z),
	\end{equation*}
	with slow and fast generators given by
	\begin{align*}
		A_{n,z}^\mathrm{slow} g(x) &:= \sum_{ab,a\neq b} n\cdot \mu_a\cdot r_n(a,b,\mu,z)\left[g(x_{a\to b}^n)-g(x)\right],\\
		A_{n,x}^\mathrm{fast} h(z) &:= \sum_i b_n^i(x,z)\partial_i h(z) + \sum_{ij}a_n^{ij}(x,z)\partial_i\partial_j h(z).
	\end{align*}
	\begin{proof}[Verification of Assumption~\ref{assumption:compact_setting:convergence-of-nonlinear-generators}]
		We have to find the slow Hamiltonian~$V_{x,p}(z)$ such that
		\begin{equation*}
			\frac{1}{n}e^{-nf}A_{n,z}^\mathrm{slow}e^{nf}\xrightarrow{n\to\infty} V_{x,\nabla f(x)}(z)
		\end{equation*}
		as specified in Assumption~\ref{assumption:compact_setting:convergence-of-nonlinear-generators}. We have
		\begin{align*}
			\frac{1}{n}e^{-nf(x)}A_{n,z}^\mathrm{slow}e^{nf(x)}&= \sum_{ab,a\neq b}\mu_a r_n(a,b,\mu,z) \left[\exp\{n(f(x_{a\to b}^n)-f(x))\}-1\right].
		\end{align*}
		Suppose that~$x_n=(\mu_n,w_n)\to x$. Then by Taylor expansion,
		\begin{equation*}
			n(f(x_{a\to b}^n)-f(x_n)) \xrightarrow{n\to\infty} \ip{\nabla f(x)}{e_{b}-e_a+e_{ab}},
		\end{equation*}
		for all~$f\in C^2(E)$, uniformly on compacts~$K\subseteq E$. By the convergence assumption on~$r_n$, we obtain the claimed convergence with~$D_0=C^2(E)$ and
		\begin{equation*}
			V_{x,p}(z) = \sum_{ab,a\neq b}\mu_a r(a,b,\mu,z)\left[e^{p_b-p_a+p_{ab}}-1\right].
		\end{equation*}
	\end{proof}
	\begin{proof}[Verification of Assumption~\ref{assumption:compact_setting:convergence-of-nonlinear-generators-external}]
		Let~$x_n\to x$ in~$E$ and~$h_n\to h$ in~$C(F)$. By the convergence assumptions on the coefficients~$b_n^i$ and~$a_n^{ij}$, we obtain
		\begin{equation*}
			A_{n,x_n}^\mathrm{fast}h_n(z) \to \sum_i b^i(x,z) \partial_ih(z) + \sum_{ij} a^{ij}(x,z)\partial_i\partial_j h(z) =: A_x^\mathrm{fast}h(z)
		\end{equation*}
		uniformly over~$z \in F$.
	\end{proof}
	\begin{proof}[Verification of Assumption~\ref{assumption:compact_setting:eigen_value}]
		The first part follows since~$F$ is compact. Let~$x\in E$ and~$p\in\mathbb{R}^d$. We aim to find a strictly positive eigenfunction~$u:F\to(0,\infty)$ and an eigenvalue~$\lambda \in \mathbb{R}$ such that
		\begin{equation}\label{SF:eq:proof-mean-field-eigenvalue-problem}
			(V_{x,p} + A_x^\mathrm{fast})u=\lambda u.
		\end{equation}
		% 	Hence part~\ref{SF:item:assumption:PI:solvePI} follows, as remarked below Assumption~\ref{assumption:compact_setting:eigen_value}.
		Equation~\eqref{SF:eq:proof-mean-field-eigenvalue-problem} is a principal-eigenvalue problem for an uniformly elliptic operator. Uniform ellipticity follows by Condition~\ref{SF:condition:mean-field:reg-coefficients} on the diffusion coefficients and the uniform convergence in Assumption~\ref{SF:mean-field:conv-coeff}. Hence there exists a unique eigenfunction~$u \in C^2(F)$ with a real eigenvalue~$\lambda$ (e.g. Sweers~\cite{Sw92}). By~\cite{DoVa75a}, this principal eigenvalue satisfies the variational representation
		\begin{equation*}
			\lambda = \sup_{\pi\in\mathcal{P}(F)}\left[\int_F V_{x,p}(z)\,\pi(\dd z) - \mathcal{I}(x,\pi)\right],
		\end{equation*} 
		with the functional
		\begin{equation*}
			\mathcal{I}(x,\pi)=-\inf_{\phi: \, \inf \phi>0}\int_F\frac{A_x^\mathrm{fast}\phi(z)}{\phi(z)}\, \pi(\dd z).
		\end{equation*}
		Hence Assumption~\ref{assumption:compact_setting:eigen_value} holds with the Hamiltonian $\cH$ as claimed.
	\end{proof}

	Assumptions \ref{assumption:results:regularity_of_V}, \ref{assumption:results:regularity_I} and \ref{assumption:Hamiltonian_vector_field} are needed to obtain existence and uniqueness of the Hamilton-Jacobi-Bellman equation. They are therefore also treated in \cite{KrSc19}; we refer to the second example of Section 5 therein.
	
	\begin{proof}[Verification of Assumption~\ref{assumption:results:regularity_of_V}]
		This assumption is proven in~\cite[Proposition~5.13]{KrSc19}. 
	\end{proof}
	
	\begin{proof}[Verification of Assumption~\ref{assumption:results:regularity_I}]
		This assumption is proven in~\cite[Proposition~5.10]{KrSc19}.
	\end{proof}
	
	\begin{proof}[Verification of Assumption~\ref{assumption:Hamiltonian_vector_field}]
		This can be carried out in a similar way as in Proposition  5.19 of \cite{KrSc19}.
	\end{proof}
	
	\subsection*{Acknowledgment}
	MS acknowledges financial support through NWO grant 613.001.552.
	
	\appendix

	\section{A more general continuity estimate} \label{section:continuity_estimate_general}
	%%%%%%%%%%
	
	The following appendix is a verbatim copy from \cite{KrSc19} that we include for completeness, as it is relevant for our main example. It is one the extension of the notion of the continuity estimate for two penalization functions instead of one.

	\begin{definition}
		We say that $\{\Psi_1,\Psi_2\}$, $\Psi_i : E^2 \rightarrow \bR^+$ is a \textit{pair of penalization functions} if $\Psi_i \in C^1(E^2)$ and if $x = y$ if and only if $\Psi_i(x,y) = 0$ for all $i$.
	\end{definition}

	\begin{definition}[Continuity estimate] \label{def:fundamental_inequality_extended}
		Let $\cG: E \times\mathbb{R}^d \times \cP(F) \rightarrow \bR$, $ (x,p,\pi)\mapsto \cG(x,p,\pi)$ be a function and $\{\Psi_1,\Psi_2\}$ be a pair of penalization functions. Suppose that for each $\varepsilon > 0$ there is a sequence $\alpha_2 \rightarrow \infty$. As before, we suppress the dependence on $\varepsilon$. Suppose that for each $\varepsilon$ and $\alpha_2$ , there is a sequence $\alpha_1 \rightarrow \infty$. We suppress writing the dependence of the sequence $\alpha_1$ on $\varepsilon$ and $\alpha_2$. We write $\alpha = (\alpha_1,\alpha_2)$.
		
		Suppose that for each triplet $(\varepsilon,\alpha_1,\alpha_2)$ as above we have variables $(x_{\varepsilon,\alpha},y_{\varepsilon,\alpha})$ in $E^2$ and measures $\pi_{\varepsilon,\alpha}$ in $\cP(F)$. We say that this collection is \textit{fundamental} for $\cG$ with respect to $\{\Psi_1,\Psi_2\}$ if:
		\begin{enumerate}[(C1)]
			\item  For each $\varepsilon$, there are compact sets $K_\varepsilon \subseteq E$ and $\widehat{K}_\varepsilon \subseteq \cP(F)$ such that for all $\alpha$ we have $x_{\varepsilon,\alpha},y_{\varepsilon,\alpha} \in K_\varepsilon$ and $\pi_{\varepsilon,\alpha} \in \widehat{K}_\varepsilon$.
			\item For each $\varepsilon > 0$ and $\alpha_2$ there are limit points $x_{\varepsilon,\alpha_2}, y_{\varepsilon,\alpha_2} \in K_\varepsilon$ of $x_{\varepsilon,\alpha}$ and $y_{\varepsilon,\alpha}$ as $\alpha_1 = \alpha_1(\varepsilon,\alpha_2) \rightarrow \infty$. For each $\varepsilon$ there are limit points $x_\varepsilon,y_\varepsilon$ in $K_\varepsilon$ of $x_{\varepsilon,\alpha_2}$ and $y_{\varepsilon,\alpha_2}$ as $\alpha_2 \rightarrow \infty$. We furthermore have
			\begin{align*}
				& \Psi_1(x_{\varepsilon,\alpha_2},y_{\varepsilon,\alpha_2}) = 0 && \forall \, \varepsilon >0, \, \forall \, \alpha_2, \\
				& \Psi_1(x_\varepsilon,y_\varepsilon) + \Psi_2(x_\varepsilon,y_\varepsilon) = 0, && \forall \, \varepsilon > 0, \\
				& \lim_{\alpha_1 \rightarrow \infty} \alpha_1 \Psi_1(x_{\varepsilon,\alpha_1,\alpha_2},x_{\varepsilon,\alpha_1,\alpha_2}) = 0, && \forall  \, \varepsilon >0, \, \forall \, \alpha_2, \\
				& \lim_{\alpha_2 \rightarrow \infty} \alpha_2 \Psi_2(x_{\varepsilon,\alpha_1},x_{\varepsilon,\alpha_1}) = 0, && \forall \, \varepsilon > 0, \\
			\end{align*}
			\item We have
			\begin{align} 
				& \sup_{\alpha_2} \sup_{\alpha_1} \cG\left(y_{\varepsilon,\alpha}, - \sum_{i=1}^k\alpha_i (\nabla \Psi_i(x_{\varepsilon,\alpha},\cdot))(y_{\varepsilon,\alpha}),\pi_{\varepsilon,\alpha}\right) < \infty, \label{eqn:control_on_Gbasic_sup_extended} \\
				& \inf_{\alpha_2} \inf_{\alpha_1} \cG\left(x_{\varepsilon,\alpha}, \sum_{i=1}^k\alpha_i (\nabla \Psi_i(\cdot,y_{\varepsilon,\alpha}))(y_{\varepsilon,\alpha}),\pi_{\varepsilon,\alpha}\right) > - \infty. \label{eqn:control_on_Gbasic_inf_extended} 	
			\end{align} \label{itemize:funamental_inequality_control_upper_bound_extended}
			In other words, the operator $\cG$ evaluated in the proper momenta is eventually bounded from above and from below.
		\end{enumerate}
		We say that $\cG$ satisfies the \textit{continuity estimate} if for every fundamental collection of variables we have for each $\varepsilon > 0$ that
		\begin{multline}\label{equation:Xi_negative_liminf_extended}
			\liminf_{\alpha_2 \rightarrow \infty} \liminf_{\alpha_1 \rightarrow \infty} \cG\left(x_{\varepsilon,\alpha},\sum_{i=1}^2 \alpha_i \nabla \Psi_i(\cdot,y_{\varepsilon,\alpha})(x_{\varepsilon,\alpha}),\pi_{\varepsilon,\alpha}\right) \\
			- \cG\left(y_{\varepsilon,\alpha},- \sum_{i=1}^2 \alpha_i \nabla \Psi_i(x_{\varepsilon,\alpha},\cdot)(y_{\varepsilon,\alpha}),\pi_{\varepsilon,\alpha}\right) \leq 0.
		\end{multline}
	\end{definition}

	%%%%%%%%%%%%%
	\bibliography{KraaijBib}
\bibliographystyle{abbrv}

	%\printbibliography
\end{document}